\documentclass[11pt]{amsart}
\usepackage{color,amssymb}
\usepackage[normalem]{ulem}

\usepackage{graphicx}
\usepackage{subcaption}
\usepackage{enumerate}

\usepackage[all]{xy}

\usepackage{tikz}
\usepackage{float}
\usepackage{soul}

\definecolor{brown}{rgb}{0.43, 0.21, 0.1}
\definecolor{green}{cmyk}{0.64,0,0.95,0.40}

\newtheorem{theorem}{Theorem}[section]
\newtheorem{claim}{Claim}
\newtheorem{question}{Question}
\newtheorem{proposition}[theorem]{Proposition}
\newtheorem{corollary}[theorem]{Corollary}
\newtheorem{lemma}[theorem]{Lemma}
\theoremstyle{definition}
\newtheorem{example}[theorem]{Example}%[section]
\newtheorem{problem}[theorem]{Problem}%[section]
\newtheorem{remark}[theorem]{Remark}%[section]
\newtheorem{construction}{Construction}%[section]
\newtheorem{definition}[theorem]{Definition}%[section]
\newtheorem*{claim*}{Claim}
\newtheorem*{theorem*}{Theorem}
\newtheorem*{definition*}{Definition}

\newcommand\w{\omega}
\newcommand{\N}{\mathbb N}
\newcommand{\Int}{\operatorname{int}}

\newcommand{\U}{\mathcal U}
\newcommand{\B}{\mathcal{B}}

\newcommand{\G}{\mathcal{G}}

\newcommand{\F}{\mathcal{F}}

\newcommand{\cl}{\operatorname{cl}}

\newcommand{\la}{\langle}
\newcommand{\ra}{\rangle}

\newcommand{\W}{\mathcal W}

\newcommand{\nothing}[1]{}

\def\N{\mathbb N}
\def\t_c{\tau_{comp}}

\def\cl{\operatorname{cl}}
\def\Int{\operatorname{Int}}
\def\F{\mathcal{F}}
\def\H{\mathcal{H}}
\def\G{\mathcal G}
\def\cl{\operatorname{cl}}
\begin{document}

\title[Open filters and measurable cardinals]
{Open filters and measurable cardinals}

%\title[Spaces admitting finitely many free open filters]
%{Spaces admitting finitely many free open filters}

\author{Serhii Bardyla, Jaroslav \v{S}upina and Lyubomyr Zdomskyy}

\address{University of Vienna, 
Institute of Mathematics, 
Kolingasse 14-16, 
1090 Vienna, 
Austria.}
\email{sbardyla@gmail.com}
\urladdr{http://www.logic.univie.ac.at/$\tilde{\ }$bardylas55/}

\address{Institute of Mathematics,
P.J. \v{S}af\'arik University in Ko\v sice, Slovakia}
\email{jaroslav.supina@upjs.sk}
\urladdr{https://umv.science.upjs.sk/supina/index.html}

\address{TU Wien,
Institute of Discrete Mathematics and Geometry
Research Unit 1 (Algebra)
TU Wien,
Wiedner Hauptstrasse 8–10/104
1040 Wien, Austria}
\email{lzdomsky@gmail.com}
\urladdr{https://dmg.tuwien.ac.at/zdomskyy/}


\thanks{The research of the first named author was funded in whole by the Austrian Science Fund FWF [10.55776/ESP399]. The~second named author was supported by the~Slovak Research and Development Agency under the Contract no. APVV-20-0045, and the~Slovak Grant Agency VEGA under the contract No. VEGA 1/0657/22. The~research was initiated during the~research stay of the~second author under the~scholarship ICM-2020-00442 of the~Austrian Agency for International Cooperation in Education and Research (OeAD-GmbH). 
The research of the third named author was funded in whole by the Austrian Science Fund FWF [10.55776/I5930].}

\begin{abstract}
In this paper, we investigate the poset $\mathbf{OF}(X)$ of free open filters on a given space $X$. In particular, we characterize spaces for which $\mathbf{OF}(X)$ is a lattice. For each $n\in\mathbb{N}$ we construct a scattered space $X$ such that $\mathbf{OF}(X)$ is order isomorphic to the $n$-element chain, which implies the affirmative answer to two questions of Mooney. Assuming CH we construct a scattered space $X$ such that $\mathbf{OF}(X)$ is order isomorphic to $(\omega+1,\geq)$. To prove the latter facts we introduce and investigate a new stratification of ultrafilters which depends on scattered subspaces of $\beta(\kappa)$. 
Assuming the existence of $n$ measurable cardinals, for every $m_0,\ldots,m_{n}\in\N$ we construct a space $X$ such that $\mathbf{OF}(X)$ is order isomorphic to $\prod_{i=0}^nm_i$.
Also, we show that the existence of a metric space possessing a free $\omega_1$-complete closed, $G_\delta$, $F_{\sigma}$ or Borel ultrafilter is equivalent to the existence of a measurable cardinal.
\end{abstract}

\maketitle

%\vspace{-3mm}

\section{Introduction and preliminaries}
We shall follow terminology from \cite{Eng} and \cite{Kun}.
In this paper, all topological spaces are assumed to be Hausdorff and infinite. 
Each natural number $n$ is identified with the set $\{0,\ldots,n-1\}$. For a set $X$, by $[X]^{<\omega}$ we denote the set of all finite subsets of $X$. 
Let $\F$ be a filter on a set $X$ and $Y\subset X$ be a positive set with respect to $\F$, then the filter $\{F\cap Y: F\in \F\}$ on $Y$ is called a {\em trace of the filter $\F$ on $Y$}. 
For any element $x$ of a space $X$ by $\mathcal{N}(x)$ we denote the filter on $X$ whose base consists of all open neighborhoods of $x$. 

Let $R$ be some property of subsets of a topological space $X$.  
A filter $\F$ on $X$ is called
\begin{itemize}
\item {\em an $R$ filter} if $\F$ possesses a base consisting of the sets with property $R$;
\item {\em an $R$ ultrafilter} if $\F$ is an $R$ filter, and for any subset $A\subseteq X$ with property $R$ either $A\in \F$ or $X\setminus A\in \F$;
%\item {\em $R$ (ultra)filter on $Y$} if the trace of $\F$ on $Y$ is an $R$ (ultra)filter;
\item  {\em free} if every $x\in X$ possesses an open neighborhood $U$ such that $X\setminus U\in \F$.
%$\bigcap\{\overline{F}: F\in \mathcal{F}\}=\emptyset$.
%\item {\em $\kappa$-complete} if for any $\lambda\in\kappa$ and subfamily $\{F_{\alpha}:\alpha\in\lambda\}\subset \mathcal{F}$ the set $\bigcap_{\alpha\in\lambda}F_{\alpha}$ belongs to $\mathcal{F}$.
\end{itemize}
If a property $R$ is closed under finite intersections, then an $R$ filter $\F$ on $X$ is an $R$ ultrafilter if and only if $\F$ is maximal with respect to the inclusion among $R$ filters on $X$.
In place of $R$ we consider properties of being open, closed, 
%zero set, cozero set, 
$G_\delta$, $F_{\sigma}$ and Borel.

Free open ultrafilters were investigated in~\cite{CP, Liu, Liu2, MF}. In particular, Carlson and Porter~\cite{CP} applied open ultrafilters to the study of maximal points and lower topologies in the poset of Hausdorff topologies on a given set. By $\beta(X)$ we denote the \v{C}ech-Stone compactification of a Tychonoff space $X$.  Open ultrafilters naturally appear in \v{C}ech-Stone compactifications. A point $y\in\beta(X)\setminus X$ is called {\em remote} if $y$ is not in the closure of any nowhere dense subset of $X$. 
%Open ultrafilters naturally appear in \v{C}ech-Stone compactifications of Tychonoff spaces. Namely, a known result of van Douwen~\cite{vD1} states that if $X$ is a non-pseudocompact Tychonoff space with a countable $\pi$-base, then $\beta(X)\setminus X$ contains a remote point, i.e., there exists $y\in\beta(X)\setminus X$ such that $y$ is not in the closure of any nowhere dense subset of $X$. 
Van Douwen~\cite{vD2} showed that a point $y\in\beta(X)\setminus X$ is remote if and only if the trace of $\mathcal N(y)$ on $X$ is an open ultrafilter. See~\cite{BD, vD1, vD2, Dow0, Dow, Dow1, Dow2, Dow3, Terad, Ter, Vek} for more about remote points and their applications to non-homogeneity and butterfly points of \v{C}ech-Stone compactifications.  
%The set of all free open ultrafilters on a space $X$ is denoted by $\mathbf{U}(X)$. 
%Recall that a filter $\mathcal{F}$ is called 

%From the~work of Terada~\cite{Terad} it can be derived that the existence of a Tychonoff space $X$ possessing a free $\w_1$-complete open ultrafilter is equivalent to the existence of a measurable cardinal.  

By $\mathbf{OF}(X)$ we denote the poset of all free open filters on a space $X$  partially ordered by the inclusion, i.e., $\mathcal{F}_1\leq \mathcal{F}_2$ if and only if $\mathcal{F}_1\subseteq \mathcal{F}_2$ for any $\F_1,\F_2\in \mathbf{OF}(X)$. 
A topological space $X$ is called {\em H-closed} if $\mathbf{OF}(X)=\emptyset$. %\js{H standing for Hausdorff?} %Let us note that a space $X$ is H-closed iff $X$ is a closed subset of any Hausdorff topological space $Y$ which contains $X$ as a subspace.
Clearly, each compact space is H-closed, but the converse is not true.
Free open filters are useful in investigating H-extensions.
A space $Y$ is called an {\em extension} of a space $X$ if $Y$ contains $X$ as a dense subspace. If, moreover, the space $Y$ is Hausdorff and H-closed, then it is called an {\em H-extension} of $X$.
H-extensions is a classical topic in General Topology. It was investigated by many authors in \cite{K,K2, K3, G,P,PV,PV2,PVV,PW,S,T,V}, also see the monograph~\cite{PWo}. Two H-extensions of a space $X$ are equivalent if there exists a homeomorphism between them fixing points of $X$. We identify equivalent H-extensions. By $\mathbf{H}(X)$ we denote the set of all H-extensions of a space $X$.

Free closed filters on topological spaces were investigated by Pelant, Simon and Vaughan~\cite{PSV}, and van Douwen~\cite{vD3}. In particular, in~\cite{PSV} it was shown that every Hausdorff (Tychonoff, resp.) non-compact space admits at least $\omega_1$ ($\omega_2$, resp.) free closed filters. Also, there was constructed a normal space possessing a unique free  closed ultrafilter. 
This motivated Pelant, Simon and Vaughan to ask about the smallest possible non-zero number of free open filters on a given Hausdorff space.

Mooney~\cite{M} constructed a Hausdorff space $X$ admitting a unique free open filter and this way answered the question of Pelant, Simon and Vaughan. Note that any space $X$ such that $|\mathbf{OF}(X)|=1$ admits a unique Hausdorff extension $Y$. Moreover, $Y$ is H-closed and $Y\setminus X$ is singleton. Also, Mooney showed that for each Bell number\footnote{For a positive integer $n$ the Bell number $B(n)$ is equal to the number of equivalence relations on an $n$-element set.} $B(n)$ there exists a space $X$ such that $|\mathbf{H}(X)|=B(n)$. This motivated him to ask the following two questions.

\begin{question}[{\cite[Question 6.8]{M}}]\label{q1}
Is there a Hausdorff space with exactly three H-extensions? Four? Other non-Bell numbers?
\end{question}

\begin{question}[{\cite[Question 6.9]{M}}]\label{q2}
Is there a Hausdorff space with two one-point H-extensions and no other H-extensions?
\end{question}

For a cardinal $\kappa$ a filter $\F$ is called {\em $\kappa$-complete} if for any $\lambda\in\kappa$ and a subfamily $\{F_{\alpha}:\alpha\in\lambda\}\subset \mathcal{F}$ the set $\bigcap_{\alpha\in\lambda}F_{\alpha}$ belongs to $\mathcal{F}$.
An uncountable cardinal $\kappa$ which possesses a nonprincipal $\kappa$-complete ultrafilter is called {\em measurable}. The existence of a measurable cardinal cannot be proved within ZFC.
The following fact is a folklore (see also \cite[Lemma 10.2]{Jec}). 
\begin{theorem}\label{classic}
There exists a nonprincipal $\omega_1$-complete ultrafilter if and only if there exists a measurable cardinal.
\end{theorem}

The latter theorem was generalized for open ultrafilters in~\cite{Liu2}.

\begin{theorem}[Liu]\label{Liu}
There exists a space $X$ which possesses a free $\omega_1$-complete open ultrafilter if and only if there exists a measurable cardinal.
\end{theorem}

This paper is organized as follows: in Section~\ref{2} we investigate spaces which possess free $\w_1$-complete closed, $G_\delta$, $F_\sigma$ or Borel ultrafilters. In particular, we show that the existence of a metric space possessing a free $\w_1$-complete closed, $G_\delta$, $F_\sigma$ or Borel ultrafilter is equivalent to the existence of a
measurable cardinal. Section~\ref{3} is devoted to general properties of open filters. There we characterize spaces for which the poset of free open filters is a lattice. In Section~\ref{4} we investigate spaces which possess linearly ordered lattices of free open filters. In particular, we affirmatively answer Questions~\ref{q1} and~\ref{q2}, as well as introduce and investigate a new stratification of ultrafilters based on scattered subspaces of $\beta(\kappa)$. In Section~\ref{5} we deal with spaces having finite non-linear lattice of free open filters. In particular, assuming the existence of $n$ measurable cardinals, for every $m_0,\ldots,m_{n}\in\N$ we construct a space $X$ such that $\mathbf{OF}(X)$ is order isomorphic to $\prod_{i=0}^nm_i$.

%In this paper we characterize spaces whose posets of free open filters are lattices.
% This way we give an affirmative answer to questions~\ref{q1} and~\ref{q2}.   We show that only complete distributive lattices can be isomorphic to $\mathbf{OF}(X)$ for some Hausdorff space $X$. %This sheds some light on Question~\ref{q3}. 
%It is proved that the existence of a space $X$ which possesses a free $\omega_1$-complete open ultrafilter is equivalent to the existence of a measurable cardinal. This way we explicitly answer Question~\ref{q4}.

\section{Completeness of different sorts of ultrafilters}\label{2}

Looking at Theorem~\ref{Liu} it is natural to ask whether similar characterization holds also for closed, $G_\delta$, $F_\sigma$ or Borel ultrafilters.

Following~\cite{Burke}, a space $X$ is called {\em screened} if 
any open cover $\mathcal V$ of a space $X$ admits a refinement $\mathcal W=\bigcup_{n\in\omega}\mathcal W_n$, where for each $n\in\omega$ the family $\mathcal W_n$ consists of pairwise disjoint open sets. By Corollary 2.4 from~\cite{Burke}, every paracompact space is screened.

\begin{proposition}\label{met}
There exists a screened space $X$ possessing a free $\omega_1$-complete closed ultrafilter if and only if there exists a measurable cardinal.
\end{proposition}

\begin{proof}
The ``if'' part of the proof follows from the fact that every discrete space is screened. 

Assume that a screened space $X$ possesses a free $\omega_1$-complete closed ultrafilter $\F$. Consider the open cover $\mathcal V=\{X\setminus F: F\in \F$, $F$ is closed$\}$.
Since the space $X$ is screened, the cover $\mathcal V$ admits a refinement $\mathcal W=\bigcup_{n\in\omega}\mathcal W_n$, where for each $n\in\omega$ the family $\mathcal W_n$ consists of pairwise disjoint open sets.  
For each $n\in\w$ let $\mathbf{W_n}=\bigcup \mathcal W_n$. Note that if $\mathbf{W_n}\notin \F$, then the maximality of $\F$ implies that the closed set $X\setminus\mathbf{W_n}$ belongs to $\F$. Consider the set $M=\{n\in\omega: \mathbf{W_n}\in\F\}$. %The following argument ensures that $M\neq \emptyset$.
If $\mathbf{W_n}\notin \F$ for all $n\in\w$, then, by the $\omega_1$-completeness of $\F$, $\emptyset=\bigcap_{n\in\omega}(X\setminus \mathbf{W_n})\in \F$ which implies a contradiction. Hence the set $M$ is not empty. 
%By the definition of the cover $\mathcal V$, there exists $n\in M$ such that $W\notin \F$ for any $W\in\mathcal W_n$;
Fix any $n\in M$.
Enumerate $\mathcal W_n$ as $\{W_{\alpha}:\alpha\in\kappa\}$. For any $F\in\F$ define $H_F=\{\alpha\in\kappa: F\cap W_{\alpha}\neq \emptyset\}$. Consider the filter $\mathcal H$ on $\kappa$ generated by the base $\{H_F:F\in\F\}$. By the definition of the cover $\mathcal V$, $W_{\alpha}\notin \F$ for any $\alpha\in\kappa$. Therefore, the filter $\mathcal H$ is nonprincipal. Since the filter $\F$ is $\omega_1$-complete, then so is $\mathcal H$. Pick any subset $A\subset \kappa$.
Assume that $A\notin\mathcal H$. Then for every $F\in \F$ the set $F\cap (\bigcup_{\alpha\in\kappa\setminus A}W_{\alpha})\neq \emptyset$. It follows that $F\cap (X\setminus (\bigcup_{\alpha\in A})W_{\alpha})\neq \emptyset$ for all $F\in \F$. Since the set $G=X\setminus (\bigcup_{\alpha\in A})W_{\alpha}$ is closed and $\F$ is a closed ultrafilter, we get that $G\in\F$. %Fix any $F\in\F$ such that $F\subset \mathbf{W_n}$. 
Then $\mathcal H\ni H_G\subset \kappa\setminus A$, witnessing that $\kappa\setminus A\in\mathcal H$.
Hence $\mathcal H$ is a nonprincipal $\omega_1$-complete ultrafilter. Theorem~\ref{classic} yields the existence of a measurable cardinal.
\end{proof}

\begin{lemma}\label{Fs}
An $\omega_1$-complete filter $\F$ on a space $X$ is an $F_{\sigma}$ ultrafilter if and only if $\F$ is a closed ultrafilter.
 \end{lemma}
 
 \begin{proof}
($\Rightarrow$) Fix an arbitrary $F\in\F$. There exists an $F_{\sigma}$ set  $B\subset F$ which belongs to $\F$. Then $B=\bigcup_{i\in\omega}B_i$ where the sets $B_i$ are closed. To derive a contradiction, assume that there exists no $i\in\omega$ such that $B_i\in\F$. Taking into account that the sets $B_i$ are closed and $\F$ is an $F_{\sigma}$ ultrafilter, we get that $X\setminus B_i\in \F$ for every $i\in\omega$. Since the filter $\F$ is $\omega_1$-complete, the set $C=\bigcap_{i\in\omega}(X\setminus B_i)$ belongs to $\F$. But $\emptyset=B\cap C\in\F$ which implies a contradiction. Hence there exists $i\in\omega$ such that $B_i\in\F$. This provides that $\F$ is a closed filter. At this point it is easy to see that $\F$ is a closed ultrafilter.

($\Leftarrow$) Fix any $F_{\sigma}$ subset $B\subset X$ such that $B\notin \F$. Find closed sets $B_{i}$, $i\in\omega$ such that $B=\bigcup_{i\in\omega}B_i$. Since $\F$ is a closed ultrafilter and $B\notin \F$, for each $i\in\omega$ there exists $F_i\in\F$ such that $F_i\cap B_i=\emptyset$. Then $F=\bigcap F_i\in \F$ and $F\cap B=\emptyset$, witnessing that $\F$ is an $F_{\sigma}$ ultrafilter.
 \end{proof}

A space $X$ is called {\em perfectly normal} if it is normal and every closed subset in $X$ is $G_{\delta}$. For a filter $\F$ on a space $X$ by $\overline{\F}$ we denote the closed filter on $X$ generated by the family $\{\overline{F}: F\in\F\}$.

\begin{lemma}\label{Gd}
Let $R$ be any property such that each closed subset of a perfectly normal space $X$ has this property. If $\F$ is a free $\omega_1$-complete $R$ ultrafilter on $X$, then $\overline{\F}$ is a free $\omega_1$-complete closed ultrafilter.
\end{lemma}

\begin{proof}
Clearly, $\overline{\F}$ is an $\omega_1$-complete closed filter. 
Consider any closed set $A$ such that $A\notin \overline{\F}$. Assuming that $A\in \F$ we obtain that $A=\overline{A}\in\overline\F$, which implies a contradiction. Thus, $A\notin \F$. Since $A$ has property $R$ and $\F$ is a $R$ ultrafilter, we obtain that $X\setminus A\in \F$. Since every open subset of $X$ is $F_{\sigma}$, there exist closed sets $B_i$, $i\in\omega$ such that $X\setminus A=\bigcup_{i\in\omega}B_i$. Bearing in mind that the sets $B_i$, $i\in\w$ have property $R$, similarly as in the proof of Lemma~\ref{Fs} one can check that there exists $i\in\omega$ such that $B_i\in\F$. Then $B_i=\overline{B_i}\in \overline \F$ and $A\cap B_i=\emptyset$, witnessing that $\overline{\F}$ is a closed ultrafilter.
 \end{proof}

Theorem~\ref{Liu}, Proposition~\ref{met} and Lemmas~\ref{Fs},~\ref{Gd} imply the following.

\begin{theorem}\label{bim}
The following statements are equivalent:
\begin{enumerate}
    \item there exists a space $X$ possessing a free $\omega_1$-complete open ultrafilter;
    \item there exists a screened space $X$ possessing a free $\omega_1$-complete closed ultrafilter;
    \item there exists a screened space $X$ possessing a free $\omega_1$-complete $F_{\sigma}$ ultrafilter;
    \item there exists a perfectly normal screened space $X$ possessing a free $\omega_1$-complete $G_{\delta}$ ultrafilter;
      \item there exists a perfectly normal screened space $X$ possessing a free $\omega_1$-complete Borel ultrafilter;
    \item there exists a measurable cardinal.
\end{enumerate}
\end{theorem}

Note that each metrizable space is perfectly normal, paracompact and, thus, screened. Theorem~\ref{bim} implies the following.

\begin{theorem}
For any property $R\in\{open, closed, G_{\delta}, F_{\sigma}, Borel\}$ there exists a metrizable space $X$ posse\-ssing a free $\w_1$-complete $R$ ultrafilter if and only if there exists a measurable cardinal.
\end{theorem}

Next we present a few examples of spaces with nice topological properties possessing a free $\omega_1$-complete closed, $G_{\delta}$ or Borel ultrafilter. 
The following simple lemma shows that this kind of examples cannot be Lindel\"{o}f.
\begin{lemma}\label{Lind}
If a space $X$ is Lindel\"{o}f, then $X$ possesses no free $\omega_1$-complete filters.
\end{lemma}

\begin{proof}
Fix any free filter $\F$ on a 
Lindel\"{o}f space $X$. Then every point $x\in X$ possesses an open neighborhood $V(x)$ such that $X\setminus V(x)\in \F$. Since $X$ is Lindel\"{o}f the open cover $\{V(x):x\in X\}$ has a countable subcover $\{V(x_n):n\in\omega\}$. Since $\bigcap_{n\in\omega}X\setminus V(x_n)=\emptyset$, the filter $\F$ is not $\omega_1$-complete. 
\end{proof}

In what follows in this section we assume that ordinals carry the order topology.
A filter $\F$ on an ordinal $\lambda$ of uncountable cofinality is called the {\em club filter} if $\F$ is generated by closed unbounded subsets of $\lambda$.  For more about club filters, see Chapter II.6 of~\cite{Kun}. By $cf(\lambda)$ we denote the cofinality of $\lambda$. The next lemma shows that Proposition~\ref{met} fails without the assumption of being screened.

\begin{lemma}\label{club}
For each ordinal $\lambda$ of uncountable cofinality the club filter is a free $cf(\lambda)$-complete closed ultrafilter.  
\end{lemma}

\begin{proof}
Fix any bounded closed subset $A\subset \lambda$. Then the set $\{\alpha\in \lambda:\alpha> \sup A\}$ is an element of the club filter disjoint with $A$. Hence the club filter is a closed ultrafilter. Since each element of $\lambda$ has a bounded open neighborhood, the club filter is free. Finally, the $cf(\lambda)$-completeness of the club filter follows from~\cite[Lemma 6.8]{Kun}.
\end{proof}

By $\mathfrak{t}$ we denote the minimal cardinality of a maximal tower on $\omega$. See~\cite{Bla} for more details. It is known that $\omega_1\leq \mathfrak t\leq \mathfrak c$. A space $X$ is called {\em sequentially compact} if each sequence in $X$ has a convergent subsequence.

\begin{proposition}
There exists a separable first-countable normal locally compact space $X$ of cardinality $\omega_1$ which possesses a free $\omega_1$-complete closed ultrafilter. 
Moreover, if $\mathfrak t=\omega_1$, then $X$ is sequentially compact.
\end{proposition}

\begin{proof}
The space $X$ will be a subspace of a space constructed by Franklin and Rajagopalan in~\cite{FR} (see also Example 7.1 in~\cite{int}). 
Fix an increasing maximal tower $\mathcal T=\{T_{\alpha}\mid \alpha\in \kappa\}$ on $\omega$.
That is $\mathcal T\subset [\omega]^{\omega}$, $T_{\alpha}\subset^* T_{\beta}$, $|T_{\beta}\setminus T_{\alpha}|=\omega$ for any $\alpha\in\beta$, and there exists no subset $P\subset \omega$ such that $|\omega\setminus P|=\omega$ and $T_{\alpha}\subset^* P$ for all $\alpha\in \kappa$. The maximality of the tower $\mathcal T$ implies that the ordinal $\kappa$ is uncountable.
Consider the space $Y(\mathcal T)=\mathcal T\cup\omega$ which is topologized as follows. Points of $\omega$ are isolated and a basic open neighborhood of $T\in \mathcal T$ has the form
$$B(S,T,F)=\{P\in\mathcal T\mid S\subset^* P\subseteq^* T\}\cup ((T\setminus S)\setminus F),$$
where $S\in \mathcal T\cup\{\emptyset\}$ satisfies $S\subset^{*}T$ and $F$ is a finite subset of $\omega$. 
By Example 7.1 from~\cite{int}, $Y(\mathcal T)$ is separable normal locally compact and sequentially compact. Note that the subspace $\mathcal T$ of $Y(\mathcal T)$ is homeomorphic to the ordinal $\kappa$.
Let $X=\omega\cup \{T_{\alpha}:\alpha\in\omega_1\}$ be a subspace of $Y(\mathcal T)$. It is easy to see that the subspace $Z=\{T_{\alpha}:\alpha\in\omega_1\}$ of $X$ is homeomorphic to the cardinal $\omega_1$. Lemma~\ref{club} implies that the subspace $Z$ possesses a free $\omega_1$-complete closed  ultrafilter $\F$. Since the set $Z$ is closed in $X$, the filter $\F$ generates a free $\omega_1$-complete closed  ultrafilter on $X$. Since $\omega$ is a dense subset of $X$, the space $X$ is separable. Being an open subspace of a locally compact space $Y$, the space $X$ remains locally compact. First-countability of $X$ follows from the definition of topology on $Y$. Consider any two closed disjoint subsets $A,B$ of $X$. If $|A|=|B|=\omega_1$, then $A\cap Z$ and $B\cap Z$ are closed unbounded. It follows that $A\cap B\neq \emptyset$ which contradicts our assumption. Hence without loss of generality we can assume that the set $A$ is countable. Let $\xi=\sup\{\alpha\in\omega_1: T_{\alpha}\in A\}$ and $R=\omega\cup\{T_{\alpha}:\alpha\leq \xi\}$. Put $C= B\cap R$. Clearly, $C$ and $A$ are disjoint closed subsets of an open countable subset $R$ of $X$. Hence there exist disjoint open subsets $U(A)$ and $U(C)$ of $R$ such that $A\subset U(A)$ and $C\subset U(C)$. Moreover, the definition of the topology on $X$ implies that we lose no generality assuming that $(U(A)\cup U(C))\cap \omega\subset T_{\xi}$. It is easy to check that the set $W=(\omega\setminus T_{\xi})\cup \{T_{\alpha}: \alpha>\xi\}$ is open. Then the open sets $U(A)$ and $U(B)=U(C)\cup W$ are disjoint and contain the sets $A$ and $B$, respectively. Hence the space $X$ is normal.   

If $\mathfrak t=\omega_1$, then we can assume that $\kappa=\omega_1$. In this case $X=Y(\mathcal T)$ is sequentially compact and satisfies other mentioned properties (see~\cite{FR} or Example 7.1 from~\cite{int} for more details).
\end{proof}

Recall that a space $X$ is called {\em countably compact} if any countable family with the finite intersection property consisting of closed subsets of $X$ has a nonempty intersection.

\begin{lemma}\label{lcc}
Each free closed filter on a countably compact space can be extended to a free $\omega_1$-complete closed ultrafilter. 
\end{lemma}

\begin{proof}
Let $X$ be a countably compact space possessing a free closed filter $\F'$. Enlarge $\F'$ to a free closed ultrafilter $\F$. By the countable compactness of $X$, the family $\{\bigcap H: H\in [\F]^{\leq \omega}\}$ generates an $\omega_1$-complete free closed filter $\H$. Taking into account the maximality of $\F$ and the inclusion $\F\subset \mathcal H$, we obtain that $\F=\mathcal H$ is an $\omega_1$-complete free closed ultrafilter. 
\end{proof}

%Lemma~\ref{lcc} does not hold for pseudocompact spaces. Consider Mrowka space.

The {\em club principle} is the following consistent with ZFC statement: There is a sequence $\{A_{\alpha}: \alpha<\omega_1, \alpha$ limit$\}$ such that for every $\alpha$, $A_{\alpha}$ is cofinal in $\alpha$ and for every uncountable $X\subset \omega_1$ the set $\{\alpha\in \omega_1: A_{\alpha}\subset X\}$ is not empty. The club principle was introduced by Ostaszewski~\cite{Ost} and usually is denoted by $\clubsuit$. 

According to \cite{ER} an uncountable space $X$ is called {\em sub-Ostaszewski} if every open subset of $X$ is either countable or co-countable. By \cite[Proposition 2.10]{ER} each sub-Ostaszewski space $X$ has size $\omega_1$. Proposition 2.2 from \cite{ER} implies that at most one point in $X$ has no countable open neighborhood. Note that the subspace $Y=\{x\in X: x$ possesses a countable open neighborhood$\}$ of $X$ is not Lindel\"of, witnessing that $X$ is not hereditary Lindel\"of. By~\cite[Proposition 2.4]{ER}, any sub-Ostaszewski space is hereditary separable. Recall that a hereditary separable regular space, which is not hereditary Lindel\"of, is called an {\em S-space}~\cite{Roitman}. By the renowned result of Todor\v{c}evi\'c~\cite{Tod}, there is a model of ZFC which contains no S-spaces and, as a consequence, no regular sub-Ostaszewski spaces.  %and has cardinality $\omega_1$. 
%By Lemma~\ref{Lind}, no Lindel\"of space possesses a free $\omega_1$-complete filter. 
%So, it is consistent with ZFC that there exist no regular sub-Ostaszewski spaces. 

%Note that every Hausdorff sub-Ostaszewski space is locally countable in all but possible one point. Removing the point we get a locally countable sub-Ostaszewski space. By~\cite{ER}, any sub-Ostaszewski space has cardinality $\omega_1$ 

\begin{lemma}\label{club1}
($\clubsuit$) There exists a perfectly normal hereditary separable first-countable locally compact space $X$ which possesses a free $\omega_1$-complete filter $\F$ which is simultaneously closed $G_{\delta}$, $F_{\sigma}$ and Borel ultrafilter.
\end{lemma}

\begin{proof}
In~\cite{Ost} it was proven that assuming $\clubsuit$ there exists a topology $\tau$ on $\omega_1$ such that $X=(\omega_1,\tau)$ is a perfectly normal first-countable locally compact sub-Ostaszewski space. Moreover, for any $\xi\in\omega_1$ the initial segment $\{\alpha: \alpha\leq \xi\}$ is open in $X$. It is straightforward to check that every Borel subset of  any sub-Ostaszewski space is either countable or co-countable.

%For each $\xi\in\omega_1$ put $[\xi,\omega_1)=\{\alpha\in\omega_1:\alpha\geq \xi\}$.
Consider the filter $\F$ generated by the family $\{\omega_1\setminus \xi: \xi\in\omega_1\}$. Clearly, the filter $\F$ is $\omega_1$-complete. Since initial segments are open in $X$, the filter $\F$ is free and closed. Consequently, the filter $\F$ is $F_{\sigma}$ and Borel.  Since the complement to each element of $\F$ is countable, $\F$ is a $G_{\delta}$ filter. Let us show that $\F$ is simultaneously closed, $F_{\sigma}$, $G_{\delta}$ and Borel ultrafilter.

Pick any Borel set $A\subset X$ such that $A\notin \F$. Since any Borel set in $X$ is either countable or co-countable, we get that the set $A$ is countable. Then $\omega_1\setminus (\sup A+1)$ is an element of $\F$ disjoint with $A$. Hence $\F$ is simultaneously closed, $F_{\sigma}$, $G_{\delta}$ and Borel ultrafilter. %Lemma~\ref{Fs} implies that $\F$ is an $F_{\sigma}$ ultrafilter. Note that $\overline{\F}=\F$. Since the space $X$ is perfectly normal, Lemma~\ref{Gd} implies that $\F$ is a $G_{\delta}$ ultrafilter.  
\end{proof}

% which admit a free $\omega_1$-complete filter. 

The example constructed in the proof of Lemma~\ref{club1} can be made countably compact~\cite{Ost}. But for this we need to assume $\diamondsuit$ instead of $\clubsuit$. For more information about $\diamondsuit$ see Chapter II.7 of~\cite{Kun}. Thus, we get the following:

\begin{proposition}\label{perf}
($\diamondsuit$) There exists a hereditary separable perfectly normal first-countable locally compact countably compact space which possesses a free $\omega_1$-complete filter $\F$ which is simultaneously closed $G_{\delta}$, $F_{\sigma}$ and Borel ultrafilter.
\end{proposition}

Weiss~\cite{W} showed that MA implies that every perfectly normal countably compact space is compact. Moreover, Nyikos and Zdomskyy~\cite{NZ} proved that under PFA every first-countable separable normal countably compact space is compact. Thus, Proposition~\ref{perf} cannot be proved within ZFC.

%We finish this chapter with an auxiliary proposition about the character of open ultrafilters.

%We start with a simple but useful lemma.

%By Corollary 2.4 from~\cite{Burke}, every paracompact space is screened. Since metrizable spaces are paracompact we obtain the following.
%\begin{corollary}
%There exists a metrizable space $X$ possessing a free $\omega_1$-complete closed ultrafilter if and only if there exists a measurable cardinal.
%\end{corollary}

\section{General facts about the poset of free open filters on a given space}\label{3}
%In this section we investigate properties of a poset $\mathbf{OF}(X)$. 
%From this point on we assume all spaces to be Hausdorff.
%Most of results of this section are auxiliary and they will be used in the next ones. 
We start with the following auxiliary lemma.

\begin{lemma}\label{unew}
An open filter $\F$ on a space $X$ is an open ultrafilter if and only if for any $F\in\F$ and open subset $U\subset F$ either $U\in \F$ or $F\setminus\overline{U}\in \F$.
\end{lemma}

\begin{proof}
($\Rightarrow$) Assuming that $\F$ is an open ultrafilter, fix any $F\in\F$ and an open subset $U\subset F$. Then either $U\in \F$ or there exists an open set $H\in\F$ such that $H\cap U=\emptyset$. In the latter case $H\cap \overline{U}=\emptyset$, implying that $\F\ni H\cap F\subseteq F\setminus \overline{U}$. Hence $F\setminus \overline{U}\in \F$.

($\Leftarrow$) Assume that $\F$ is an open filter on $X$ such that for any $F\in\F$ and any open subset $U\subset F$ either $U\in \F$ or $F\setminus\overline{U}\in \F$. Since $X\in\F$, for any open subset $U\subset X$ either $U\in \F$ or $X\setminus\overline{U}\in \F$. Then for any open set $U\notin\F$, the set $X\setminus \overline{U}$ is an element of $\F$ disjoint with $U$. Hence $\F$ is an open ultrafilter.
\end{proof}

Lemma~\ref{unew} implies that every open ultrafilter on a space $X$ contains all dense open subsets of $X$.

The set $\mathbf{H}(X)$ of all H-extensions of a space $X$ carries a natural partial order $\leq$ defined as follows: for any $Y,Z\in \mathbf{H}(X)$, $Y\leq Z$ if there exists a continuous surjection $f:Z\rightarrow Y$ such that $f(x)=x$ for any $x\in X$. It can be checked~\cite{M} that $\mathbf{H}(X)$ is a complete upper semilattice with respect to this order, i.e. each subset $A\subset \mathbf{H}(X)$ has supremum in $\mathbf{H}(X)$. Order properties of compactifications and H-extensions were investigated in~\cite{CF, K1, MRW, Men, MR, R}.

%Next we recall the definition of the Katetov H-extension.

\begin{definition}\label{Katetov}
The Katetov H-extension $K(X)$ of a space $X$ is the set 
 $$X\cup \{\F:\F\hbox{ is a free open ultrafilter on }X\}$$ endowed with the topology $\tau$ which satisfies the following conditions:
\begin{itemize}
\item A subset $A\subseteq X$ is open in $(K(X),\tau)$ if and only if $A$ is open in $X$;
\item open neighborhood base at $\F\in K(X)\setminus X$ consists of the sets $\{\F\}\cup F$, where $F\in\F$ is open.
\end{itemize}
\end{definition}
For any space $X$, the Katetov extension $K(X)$ is the supremum of $\mathbf{H}(X)$ and it can be considered as a generalization of \v{C}ech-Stone compactification (see~\cite{PV} or ~\cite{PVV} for more details).

A space $X$ is called 
\begin{itemize}
\item {\em locally H-closed} if each point $x\in X$ possesses an open neighborhood $U$ such that $\overline{U}$ is H-closed.
\item {\em almost H-closed} if it admits a unique free open ultrafilter.
\end{itemize}

A simple example of an almost H-closed space is the subspace $K(X)\setminus \{\F\}$ of $K(X)$, where $X$ is not H-closed, and $\F$ is a free open ultrafilter on $X$.
The following lemma is a consequence of
Theorem~4.1 from~\cite{G}.
\begin{lemma}\label{new}
Every almost H-closed space is locally H-closed.
\end{lemma}

Let $X$ be a locally H-closed non-H-closed space. By $\F_{\inf}$ we denote the filter generated by the base $\{X\setminus K: K$ is an H-closed subspace of $X\}$. Observe that $\F_{\inf}$ is well-defined, as the union of finitely many H-closed subspaces remains H-closed. 

\begin{lemma}\label{inf}
For any locally H-closed non-H-closed space $X$ the filter $\F_{\inf}$ is the infimum of $\mathbf{OF}(X)$. 
\end{lemma}

\begin{proof}
Clearly, $\F_{\inf}$ is a free open filter on $X$. To derive a contradiction, assume that there exists a free open filter $\mathcal H$ on $X$ such that $\F_{\inf}\not\subseteq \mathcal H$. Then there exists an H-closed subset $K$ of $X$ such that $H\cap K\neq \emptyset$ for any $H\in\mathcal H$. Then the trace of $\mathcal H$ on $K$ is a free open filter, which contradicts the H-closedness of $K$.
\end{proof}

%By~\cite{P}, for each locally H-closed non-H-closed space $X$ the poset $\mathbf{H}(X)$ contains the infimum, which now is called the Obreanu-Porter extension $OP(X)$~\cite{Her}. The extension $OP(X)$ has a singleton remainder $y$ and open neighborhood base $\mathcal B(y)$ at $y$ consists of the sets $\{y\}\cup X\setminus K$, where $K$ runs over H-closed subspaces of $X$. It is worth to note that the open filter $\mathcal F_{\inf}$ on $X$ generated by the family $\{X\setminus K: K$ is an H-closed subspace of $X\}$ is the infimum of the poset $\mathbf{OF}(X)$.

\begin{lemma}\label{op1}
Let $X$ be a locally H-closed space and $\mathcal{T}$ be a nonempty subset of $\mathbf{OF}(X)$. Then $\inf \mathcal{T}=\bigcap \mathcal{T}$.
\end{lemma}

\begin{proof}
It is straightforward to check that  the intersection of an arbitrary family of open filters is an open filter.
Hence to prove that $\inf \mathcal{T}=\bigcap \mathcal{T}$, it suffices to show that the open filter $\bigcap \mathcal{T}$ is free. 
%Since the space $X$ is almost H-closed, there exists a unique free open ultrafilter $\mathcal{F}_{\sup}$ which contains every element of $\mathcal T$.  
%By Lemma~\ref{new}, $X$ is locally H-closed. 
Lemma~\ref{inf} implies that the filter $\F_{\inf}$ is the infimum of $\mathbf{OF}(X)$. 
%It is easy to check that for every free open filter $\F$ on $X$ and H-closed subspace $K$ of $X$, the set $X\setminus K\in \F$. Then, taking into account that $X$ is locally H-closed, the open filter $\mathcal F_{\inf}$ generated by the family $\{X\setminus K: K$ is a H-closed subspace of $X\}$ is free and  coincides with the infimum of the poset $\mathbf{OF}(X)$. 
Hence $\mathcal F_{\inf}\subseteq \mathcal F$ for each $\F\in \mathcal T$. Since the filter $\F_{\inf}$ is free we get that so is $\bigcap \mathcal T$.
%To prove that the filter $\bigcap \mathcal T$ is open fix any $F\in \bigcap\mathcal{T}$. Since the set $\mathcal T$ consists of open filters, for each $\F\in \mathcal T$ there exists an open set $O_{\F}\in\F$ such that $O_{\F}\subset F$. Then $H=\bigcup_{\F\in \mathcal T}O_{\F}$ is open, $H\in \bigcap \mathcal T$ and $H\subset F$, witnessing that $\bigcap\mathcal T$ is a free open filter.
\end{proof}

\begin{lemma}\label{op}
Let $X$ be an almost H-closed space and $\mathcal{T}$ be a nonempty subset of $\mathbf{OF}(X)$. Then $\sup\mathcal T$ is generated by the family $\bigcup \mathcal{T}$.
\end{lemma}

\begin{proof}
Let $\F_{\sup}$ be the supremum of $\mathbf{OF}(X)$.
Since for each $\F\in \mathcal T$, $\F\subset\mathcal F_{\sup}$, the family $\bigcup \mathcal{T}$ has the finite intersection property. Let $\mathcal H$ be the filter generated by the family $\bigcup \mathcal{T}$. Since each filter $\F\in \mathcal T$ is free and $\F\subset \mathcal H$, the filter $\mathcal H$ is free as well. It remains to show that the filter $\mathcal H$ is open. For this, fix any $H\in \mathcal H$. It follows that there exist $\F_0,\ldots, \F_{n-1}\in \mathcal T$ and elements $F_{i}\in \mathcal F_{i}$, $i\in n$ such that $\bigcap_{i\in n}F_{i}\subseteq H$. Since the filters $\F_i$, $i\in n$ are open, for every $i\in n$ there exists an open set $O_{i}\in\F_{i}$ such that $O_i\subset F_i$. Observe that the open set $\bigcap_{i\in n}O_{i}\subseteq H$ belongs to $\mathcal H$.
\end{proof}

Since the increasing union of filters is a filter, Lemma~\ref{op} implies the following.

\begin{corollary}\label{uni}
Let $\mathcal{T}$ be a nonempty linearly ordered subset of $\mathbf{OF}(X)$. Then  $\sup\mathcal T=\bigcup \mathcal{T}$.
\end{corollary}

%\begin{proof}
%By Lemma~\ref{op}, it suffices to show that $\bigcup \mathcal{T}$ is a filter. Fix any finite subset $\{F_i:i\in n\}\subset \bigcup \mathcal{T}$. By the linearity of $\mathcal T$, there exists $\F\in\mathcal T$ such that $\{F_i:i\in n\}\subset \F$. It follows that $\bigcap_{i\in n}F_i\in \F\subset \bigcup \mathcal{T}$. Hence $\bigcup \mathcal{T}$ is a filter.
%\end{proof}

A lattice $(L,\vee,\wedge)$ is called {\em distributive} if 
$x\wedge(y\vee z)=(x\wedge y)\vee(x\wedge z)$. It is well-known that the latter condition is equivalent to its dual: 
$x\vee(y\wedge z)=(x\vee y)\wedge(x\vee z)$.

\begin{theorem}\label{char}
For a Hausdorff space $X$ the following conditions are equivalent:
\begin{enumerate}
\item $X$ is almost H-closed;
\item the posets $\mathbf{H}(X)$ and $\mathbf{OF}(X)$ are order isomorphic;
\item the poset $\mathbf{OF}(X)$ is a complete distributive lattice;
\item the poset $\mathbf{OF}(X)$ is a lattice.
\end{enumerate}
\end{theorem}

\begin{proof}
(1)$\Rightarrow$(2). Let us show that any proper extension of $X$ has a singleton reminder. To derive a contradiction, assume that $Y$ is an extension of $X$ such that $Y\setminus X$ contains two distinct points $a,b$. By Zorn's Lemma, the traces of the filters $\mathcal{N}(a)$ and $\mathcal{N}(b)$ on $X$ can be enlarged to open ultrafilters $\F_a$ and $\F_b$, respectively. Since the space $Y$ is Hausdorff, there exists $U_a\in \mathcal{N}(a)$ and $U_b\in\mathcal{N}(b)$ such that $U_a\cap U_b=\emptyset$. Therefore, $\F_a$ and $\F_b$ are distinct filters, which contradicts the almost H-closedness of $X$. Define $\phi: \mathbf{H}(X)\rightarrow \mathbf{OF}(X)$ by $\phi(Y)=\F_y$, where $Y\setminus X=\{y\}$ and $\F_y$ is the trace of the filter $\mathcal{N}(y)$ (whose base consists of open neighborhoods of $y$) on $X$. Clearly, the filter $\F_y$ is open and free, witnessing that the map $\phi$ is well-defined.
At this point it is straightforward to check that the map $\phi$ is injective. Let us check that $\phi$ is an order homomorphism. Fix any two H-extensions $Y_1$ and $Y_2$ of $X$ such that $Y_1\leq Y_2$. Let $\{y_1\}=Y_1\setminus X$ and $\{y_2\}=Y_2\setminus X$. There exists a continuous surjective map $h: Y_2\rightarrow Y_1$ fixing elements of $X$. Obviously, $h(y_2)=y_1$. Since the map $h$ is continuous, $\phi(Y_1)\subseteq \phi(Y_2)$. Hence $\phi$ is an injective order homomorphism between posets $\mathbf{H}(X)$ and $\mathbf{OF}(X)$.
Fix any free open filter $\F$ on $X$. Let $\tau$ be the topology on the set $X\cup\{\F\}$ which satisfies the following conditions:
\begin{itemize}
\item A subset $A\subseteq X$ is open in $(X\cup\{\F\},\tau)$ if and only if $A$ is open in $X$;
%\item $X$ is an open subspace of $(X\cup\{\F\},\tau)$;
\item the family $\{\{\F\}\cup F: F\in \F$ and $F$ is open$\}$ forms a base at $\F$.
\end{itemize}
Then $Y=(X\cup\{\F\},\tau)\in \mathbf{H}(X)$ and $\phi(Y)=\F$ which implies that the map $\phi$ is surjective.

(2)$\Rightarrow$ (1).  Assume that the space $X$ possesses two distinct free open ultrafilters $\F_1$, $\F_2$. Then there exist $F_1\in\F_1$ and $F_2\in \F_2$ such that $F_1\cap F_2=\emptyset$. It follows that there exists no open filter $\mathcal H$ on $X$ such that $\F_1\subset \H$ and $\F_2\subset \H$. Consequently, the poset $\mathbf{OF}(X)$ does not have a supremum. On the other hand, the Katetov extension $K(X)$ is the supremum of $\mathbf{H}(X)$ (see~\cite{PV}). Hence the posets $\mathbf{OF}(X)$ and $\mathbf{H}(X)$ cannot be order isomorphic.

(2)$\Rightarrow$(3). Since (1)$\Leftrightarrow$(2) the space $X$ is almost H-closed. Lemma~\ref{new} implies that $X$ is locally H-closed.  By the previous arguments, we have that each H-closed extension of $X$ has a singleton remainder. Proposition 1 from~\cite{Her} provides that the set of all H-closed extensions with a singleton remainder of a space $X$ forms a complete sublattice of $\mathbf{H}(X)$. Hence the lattice $\mathbf{OF}(X)$ is complete. Consider any free open filters $\mathcal F$, $\mathcal G$, $\mathcal H$ on $X$. Lemmas~\ref{op1} and~\ref{op} imply that the filter $\mathcal F\wedge(\mathcal G\vee\mathcal H)$ is generated by the family $\mathcal A=\{F\cup(G\cap H):F\in\mathcal F, G\in\mathcal G, H\in\mathcal H\}$, and the filter $(\mathcal F\wedge \mathcal G)\vee(\mathcal F\wedge \mathcal H)$ is generated by the family $\mathcal B=\{(F\cup G)\cap(F\cup H):F\in\mathcal F, G\in\mathcal G, H\in\mathcal H\}$. Obviously, $\mathcal A=\mathcal B$ witnessing that the lattice $\mathbf{OF}(X)$ is distributive.

%(2)$\Rightarrow$(3). Lemma~\ref{lat} ensures that the posets $H_1(X)$ and $\mathrm{OF(X)}$ are complete lattices.

%(1)$\Rightarrow$(2). Since the space $X$ is almost H-closed, all extensions of $X$ are H-closed and have a singleton reminder. Fix any such H-closed %extension $Y=X\cup\{y\}$. For each $x\in X$ there exists an open neighborhood $U$ of $x$ such that $y\notin \overline{U}$. Since the closure of open %set in the H-closed space is H-closed we get that $\overline{U}$ is an H-closed neighborhood of $x$. Hence $X$ is locally H-closed. By %Lemma~\ref{lat}, the poset $H_1(X)$ is a complete lattice. It remains to observe that $\mathrm{OF(X)}$ is isomorphic to the poset $H_1(X)$.

The implication (3)$\Rightarrow$(4) is trivial.

(4)$\Rightarrow$(1). Assume that $\mathbf{OF}(X)$ is a nonempty lattice and nonetheless $X$ admits two distinct free open ultrafilters $\F_1, \F_2$. Then there exists $F_1\in\F_1$ and $F_2\in\F_2$ such that $F_1\cap F_2=\emptyset$. Since $\mathbf{OF}(X)$ is a lattice
there exists a filter $\mathcal{G}\in\mathrm{OF}(X)$ such that $\F_1\cup \F_2\subset \mathcal{G}$. It follows that $\emptyset= F_1\cap F_2\in \mathcal G$, which implies a contradiction.
\end{proof}

Recall that a lattice is distributive if and only if it contains as a sublattice none of the following two lattices:

\begin{figure}[H]
\begin{center}
\begin{tikzpicture}[]
\node[circle,draw] (a) at (0, 0) {};
\node[circle,draw] (b) at (1.5, 0) {};
\node[circle,draw] (c) at (3, 0) {};
\node[circle,draw] (d) at (1.5, -1.5) {};
\node[circle,draw] (e) at (1.5, 1.5) {};
\node[circle,draw] (f) at (7, 0.4) {};
\node[circle,draw] (g) at (7, -0.4) {};
\node[circle,draw] (h) at (8.5, -1.5) {};
\node[circle,draw] (i) at (8.5, 1.5) {};
\node[circle,draw] (j) at (10, 0) {};
\foreach \from/\to in {a/d,b/d,c/d,a/e,b/e,c/e,f/g,g/h,f/i,j/h,j/i} \draw [-] (\from) -- (\to);
\end{tikzpicture}
\end{center}
\caption{Non-distributive lattices.}
\label{diagram:1}
\end{figure}

The latter fact and Theorem~\ref{char} imply the following: % corollary which sheds some light on Question~\ref{q3}.

\begin{corollary}\label{fin}
There exist five-element lattices which are not order isomorphic to $\mathbf{OF}(X)$ for any space $X$. 
\end{corollary}

Corollary~\ref{fin} implies the following natural problem, which will be studied in the next sections.

\begin{problem}\label{pr1}
Characterize finite lattices which can be represented as the lattice $\mathbf{OF}(X)$ for some Hausdorff space $X$.
\end{problem}

\begin{lemma}\label{dens}
Let $X$ be an almost H-closed space and $\mathcal{F}_1,\mathcal{F}_2\in\mathbf{OF}(X)$. Then $\Int(\overline{F})\in\mathcal{F}_2$ for each $F\in\mathcal{F}_1$.
\end{lemma}

\begin{proof}
By Theorem~\ref{char}, the poset $\mathbf{OF}(X)$ is a complete lattice. Let $\mathcal F_{\sup}$ and $\mathcal{F}_{\inf}$ be the supremum and infimum, respectively, of $\mathbf{OF}(X)$. 
Note that for any $F\subset X$ and open filter $\mathcal F$ on $X$, $\overline{F}\in\mathcal F$ if and only if $\Int(\overline{F})\in\F$.
Taking into account the latter argument, it suffices to check that for each $F\in\F_{\sup}$, $\overline{F}\in\mathcal F_{\inf}$. To derive a contradiction, assume that there exists $F\in \F_{\sup}$ such that $\overline{F}\notin \F_{\inf}$. It follows that for each $H\in \F_{\inf}$ the set $H\setminus \overline{F}$ is nonempty. Then the family $\{H\setminus \overline{F}: H\in\F_{\inf}\}$ forms a base of some free open filter on $X$ which is not contained in the filter $\F_{\sup}$, contradiction.
\end{proof}

\begin{corollary}\label{col}
Let $X$ be an almost H-closed space and $\mathcal U$ be the unique free open ultrafilter on $X$. Then the family $\{\Int(\overline{U}):U\in\mathcal U\}$ is a base of the filter $\F_{\inf}$.
\end{corollary}

\begin{proof}
Clearly, the filter $\mathcal W$ generated by the family $\{\Int(\overline{U}):U\in\mathcal U\}$ is a free open filter. %Note that for each $U\in\mathcal U$, $\overline{U}=\overline{\Int(\overline{U})}$. Since the filter $\mathcal U$ is free, the filter $\mathcal W$ is free as well. 
Lemma~\ref{dens} provides that the filter $\mathcal W$ is contained in any other free open filter on $X$. Lemma~\ref{inf} implies that $\mathcal W=\F_{\inf}$.
\end{proof}

The following lemma is useful for the investigation of finite and linear posets of free open filters.

\begin{lemma}\label{lm}
Let $\F, \H$ be two distinct free open filters on a space $X$ such that $\F\subset \H$ and there exists no open filter $\mathcal G$ satisfying $\F\subsetneqq\mathcal G\subsetneqq \H$. Let $P$ be any open set which belongs to $\H\setminus \F$. By $\mathcal T_P$ we denote the  filter on $X$ generated by the family $\{F\cap(X\setminus P): F\in \F\}$. Then the following statements hold:
\begin{enumerate}
\item $\H$ is generated by the family $\F\cup\{P\}$;
\item $\F=\{T\cup H: T\in \mathcal T_P$ and $H\in \H\}$;
\item $\mathcal T_P$ traces a free open ultrafilter on $X\setminus P$;
\item if for any open filter $\mathcal G$ on $X$ the inclusion $\F\subsetneqq \mathcal G$ implies that $\H\subset \mathcal G$, then $\F$ is generated by the family $\{U\subset X: U$ is open and $U\cap (X\setminus P)\in \mathcal T_P\}$.
\end{enumerate}

If, moreover, the space $X$ is almost H-closed, then

\begin{itemize}
\item[(5)] $\mathcal T_P$ traces an open ultrafilter on $\overline{P}\setminus P$. % generated by the family $\{F\cap(\overline{P}\setminus P): F\in \F\}$.
%\item[(6)] the filter $\mathcal T_P$ is generated by the family $\{F\cap (\overline{H}\setminus P): F\in \F, H\in\H\}$.
%\item[(6)] the filter $\mathcal T_P$ is generated by the family $\{F\cap (\overline{H}\setminus H): F\in \F, P\supset H\in\H\}$;
\end{itemize}
\end{lemma}

\begin{proof}
1. Observe that the filter $\mathcal G$ generated by the family $\F\cup\{P\}$ is open and satisfies $\F\subset\mathcal G\subset\H$, which implies that either $\mathcal G=\F$ or $\mathcal G=\H$. Since $P\notin \F$, we deduce that $\mathcal G=\H$.

2. Let $\mathcal Z$ be the filter generated by the family $\{T\cup H: T\in \mathcal{T}_P, H\in \mathcal{H}\}$. Consider any $Z\in \mathcal Z$. By the definition of $\mathcal Z$, there exist $T\in \mathcal{T}_P$ and $H\in \mathcal{H}$ such that $T\cup H\subset Z$. By item (1), the filter $\mathcal{H}$ is generated by the family $\mathcal{F}\cup\{P\}$. It follows that there exists an element $F\in \mathcal F$ such that $F\cap P\subset H$. By the definition of $\mathcal T_P$, there exists $G\in\F$ such that $G\subset F$ and $G\cap (X\setminus P)\subset T$. The choice of $F$ implies that $G\cap P\subset F\cap P\subset H$. Thus, $G\subset T\cup H\subset Z$ witnessing that $\mathcal Z\subset\mathcal{F}$.
To show the converse inclusion, fix any $F\in\mathcal F$. Put $H=F\cap P\in\mathcal H$ and $T=F\cap (X\setminus P)\in \mathcal T_P$. Obviously, $T\cup H\in \mathcal{Z}$ and $T\cup H= F$ witnessing that $\mathcal{Z}=\mathcal{F}$.

3. Clearly, $\mathcal T_P$ traces a free open filter on $X\setminus P$. To derive a contradiction, assume that the trace of $\mathcal T_P$ on $X\setminus P$ is properly contained in an open filter $\mathcal U$ on $X\setminus P$. It follows that there exists a set $U\in \mathcal U\setminus \mathcal T_P$ which is open in $X\setminus P$. Then there exists an open set $V\subset X$ such that $V\cap (X\setminus P)=U$. Let $W=V\cup P$. Consider the filter $\mathcal Z$ generated by the family $\F\cup\{W\}$. Since the set $W$ is open we get that $\mathcal Z$ is an open filter. By the definition of $\mathcal Z$, $\F\subsetneqq \mathcal Z\subsetneqq \H$ which implies a contradiction.

4. Assume that for any open filter $\mathcal G$ on $X$ the inclusion $\F\subsetneqq \mathcal G$ implies that $\H\subset \mathcal G$. Consider the filter $\mathcal Z$ which is generated by the family $\{U\subset X: U$ is open and $U\cap (X\setminus P)\in \mathcal T_P\}$. Obviously the filter $\mathcal Z$ is open and $\F\subset \mathcal Z$. If $\mathcal Z\neq \F$, the assumption of the statement implies that $\H\subset \mathcal Z$. But this is not possible, since $P\notin \mathcal Z$. Hence $\mathcal Z=\F$.

5. By Lemma~\ref{dens}, $\overline{P}\setminus P\in\mathcal T_P$. At this point item (5) follows from item (3).
\end{proof}

%For an element $p$ of a poset $P$ let
%$${\uparrow} p=\{x\in P: p\leq x\}\qquad\hbox{and}\qquad {\downarrow}p=\{x\in P: x\leq p\}.$$

A {\em character} of a filter $\F$ is the cardinal $\chi(\F)=\min\{|\mathcal B|: \mathcal B$ is a base of $\F\}$.
Recall that $$\mathfrak{u}=\min\{\chi(\F): \F\hbox{ is a nonprincipal ultrafilter on }\w\}.$$
A space $X$ is called {\em ccc} if every pairwise disjoint family of open subsets of $X$ is countable. 
We finish this chapter with an auxiliary proposition about the character of open ultrafilters.
%A {\em cellularity} of a topological space $X$ is the cardinal $$c(X)=\sup\{|\mathcal A|:\mathcal A \hbox{ is a family of open pairwise disjoint subsets of } X\}.$$ 

%For a cardinal $\kappa$, let 
%$$\mathfrak{u}_{\kappa}=\min\{\chi(\F): \F\hbox{ is an ultrafilter on }\kappa \hbox{ and }|F|=\kappa \hbox{ for any } F\in\F\}.$$  

\begin{proposition}
Let $\F$ be a free open ultrafilter on a ccc space $X$. Then $\chi(\F)\geq\mathfrak{u}$.
\end{proposition}

\begin{proof}
First we construct a special countable family of pairwise disjoint open subsets of $X$. Let $U_0$ be any open subset of $X$ which doesn't belong to $\F$. Assume that we already constructed pairwise disjoint open subsets $U_\alpha$, $\alpha\in \xi\in\w_1$ such that for any $\delta\in\xi$, $\bigcup_{\alpha\in\delta}U_{\alpha}\notin\F$. If $\bigcup_{\alpha\in\xi}U_{\alpha}\in \F$, then $\{U_{\alpha}:\alpha\in\xi\}$ is the desired family. Otherwise, an open set $T=X\setminus \overline{\bigcup_{\alpha\in\xi}U_{\alpha}}$ is nonempty, as it belongs to $\F$. Pick any $x\in T$. Since the filter $\F$ is free, there exists an open neighborhood $U(x)\subset T$ of $x$ such that $U(x)\notin\F$. Then put $U_{\xi}=U(x)$. Note that the ccc property of the space $X$ ensures that our induction will stop at some countable limit step, i.e. there exists a countable limit ordinal $\theta$ such that $\bigcup_{\alpha\in\theta}U_{\alpha}\in \F$. Enumerate the family $\{U_{\alpha}:\alpha\in\theta\}$ as $\{V_n:n\in\w\}$. For any $F\in\F$ consider the set $H_F=\{n\in\w: F\cap V_n\neq \emptyset\}$. Observe that $\mathcal H=\{H_F: F\in\F\}$ is the filter on $\w$. Pick any subset $A\subset \w$. Lemma~\ref{unew} implies that either $\bigcup_{n\in A}V_n\in\F$ or $\bigcup_{n\in\w\setminus A}V_n\in\F$. In the first case the set $\bigcup_{n\in A}V_n\in\F$ is a witness for $A\in\mathcal H$. In the second case the set $\bigcup_{n\in\omega\setminus A}V_n\in\F$ is a witness for $\w\setminus A\in \mathcal H$. Hence $\mathcal H$ is an ultrafilter. Observe that for any base $\mathcal B$ of $\F$ the family $\{H_B:B\in\mathcal B\}$ is a base of $\mathcal H$. Since $\chi(\mathcal H)\geq \mathfrak{u}$ we get that $\chi(\F)\geq \mathfrak u$.
\end{proof}

%\section{Answering questions~\ref{q1} and~\ref{q2}}
\section{Linearly ordered lattices of free open filters}\label{4}
%In this section ordinals are assumed to carry the order topology if the converse not stated.
In what follows, ordinals are assumed to carry the discrete topology, unless stated otherwise.
%The \v{C}ech-Stone compactification of a space $X$ is denoted by $\beta(X)$. 
If $X$ is a discrete space, then $\beta (X)$ can be described as the set of all ultrafilters on $X$ endowed with the topology $\tau$ generated by the base $\mathcal B=\{\langle A\rangle: A\subseteq X\}$, where $\langle A\rangle=\{u\in\beta(X): A\in u\}$. Recall that each element $x\in X$ is identified with the corresponding principal ultrafilter. Further we refer only to the aforementioned representation of $\beta(X)$.
%For a cardinal $\kappa$ we shall identify a point $\alpha\in\kappa\subset \beta(\kappa)$ with the corresponding principal ultrafilter.
%The family of infinite subsets of $\omega$ is denoted by $[\omega]^{\omega}$.
A space $X$ is called {\em scattered} if each nonempty subspace of $X$ contains an isolated point. A {\em height} of a space $X$ is the minimal ordinal $ht(X)$ such that the $ht(X)$-th Cantor-Bendixson derivative of $X$ is empty.
Recall that Cantor-Bendixson derivatives of a scattered space $X$ are defined by transfinite induction as follows, where $X'$ is the set of all accumulation points of $X$:
\begin{itemize}
\item $X^{0}=X$;
\item $X^{\alpha+1}=\big(X^{\alpha}\big)'$;
\item $X^{\alpha}=\bigcap_{\beta<\alpha}X^{\beta}$, if $\alpha$ is a limit ordinal.
\end{itemize}
The set $X^{\alpha}\setminus X^{\alpha+1}$ is called the $\alpha$-th {\em Cantor-Bendixson level} of $X$ and is denoted by $X^{(\alpha)}$.

A subset $X$ of a space $Y$ is called {\em strongly discrete} if there exists a family $(U_{x})_{x\in X}$ of open pairwise disjoint subsets of $Y$ such that $U_{x}\cap X=\{x\}$ for every $x\in X$.
The following lemma is a folklore.
% and it can be proved similarly as Theorem~3.6.14 from~\cite{Eng}.
Nonetheless, we give an easy proof of it.

\begin{lemma}\label{d}
Let $X=\{x_{\alpha}\}_{\alpha\in\lambda}$ be a strongly discrete subset of $\beta(\kappa)$. Then there exists a homeomorphism $h:\overline{X}\to\beta(\lambda)$ such that $h(x_{\alpha})=\alpha$ for each $\alpha\in\lambda$.
\end{lemma}

\begin{proof}
Since the set $X$ is strongly discrete there exists a family $\{A_{\alpha}:\alpha\in\lambda\}$ of open pairwise disjoint subsets of $\beta(\kappa)$ such that $X\cap A_{\alpha}=\{x_{\alpha}\}$ for any $\alpha\in\lambda$. 
Fix any function $f:X\rightarrow [0,1]$. Define the function $f_1:\kappa\rightarrow [0,1]$ as follows:
\begin{equation*}
    f_1(\xi)=
    \left\{
      \begin{array}{cl}
        f(x_{\alpha}), & \hbox{if~~} \xi\in A_{\alpha}\cap \kappa;\\
        0,   & \hbox{if~~} \xi\in \kappa\setminus(\bigcup_{\alpha\in\lambda}A_{\alpha}).
      \end{array}
    \right.
\end{equation*}

By the definition of $\beta(\kappa)$, there exists a continuous extension $\hat{f_1}:\beta(\kappa)\rightarrow [0,1]$ of $f_1$. Let $g=\hat{f_1}{\restriction}_{\overline{X}}$. Since $\hat{f_1}(x_{\alpha})=f(x_{\alpha})$, $g$ is a continuous extension of $f$.  Hence every function $f:X\rightarrow [0,1]$ can be extended to a continuous function $g:\overline{X}\rightarrow [0,1]$. Corollary 3.6.3 from~\cite{Eng} yields a homeomorphism $h: \overline{X}\to \beta(X)$ which fixes points of $X$. Since $X$ is a discrete set of cardinality $\lambda$, we can identify $\beta(X)$ and $\beta(\lambda)$.
\end{proof}

Let $\rho$ be an equivalence relation on a set $X$. Then for each $x\in X$ such that $x\rho y$ if and only if $x=y$, we denote the equivalence class $\{x\}$ simply by $x$.  
The following scheme was invented by Mooney~\cite{M} and it will be crucial in constructing almost H-closed spaces with certain properties of their lattices of free open filters.

\begin{construction}[Mooney~\cite{M}]\label{cons}
Let $X$ be a non-H-closed topological space, $x^*$ be a non-isolated point of $X$, and $K(X)$ be the Katetov extension of $X$ (see Definition~\ref{Katetov}).

Consider the Tychonoff product $K(X){\times}\{0,1\}$ where the set $\{0,1\}$ carries the discrete topology. Let $Y$ be the quotient space $(K(X){\times}\{0,1\})/\rho$ where the equivalence relation $\rho$ is defined as follows: $(a,i)\sim (b,j)$ if and only if $a=b$ and $i=j$, or $a=b\in K(X)\setminus X$. Finally, put $\mathrm{M}(X,x^*)=Y\setminus\{(x^*,0)\}$.
One can easily check that the space $\mathrm{M}(X,x^*)$ is scattered if and only if $X$ is scattered.
\end{construction}

An open filter $\F$ on a space $X$ is called {\em regular open} if the family $\{\Int(\overline{F}): F\in\F\}$ forms a base of $\F$. 

\begin{proposition}\label{supernew}
Let $\F$ be a free regular open filter on a space $X$. Then there exists a space $Y$ such that $\mathbf{OF}(Y)$ is order isomorphic to the subposet $\{\G\in \mathbf{OF}(X): \F\subseteq \G\}$ of $\mathbf{OF}(X)$.
\end{proposition}

\begin{proof}
By $X_\F$ we denote the space $X$ with an attached singleton $\F$, whose open neighborhood base consists of the sets $F\cup\{\F\}$, where $F\in\F$ is open. Let $Y=\mathrm{M}(X_\F,\F)$ (see Construction~\ref{cons}). 

Consider any free open filter $\Phi$ on $Y$. Note that the subspace $K(X_\F){\times}\{1\}$ of $Y$, being homeomorphic to $K(X_\F)$, is H-closed. So, $Y\setminus (K(X_\F){\times}\{1\})=X{\times}\{0\}\in\Phi$.
This allows us to define the map $g:\mathbf{OF}(Y)\to \mathbf{OF}(X)$ by $g(\Phi)=\{G\subseteq X: G{\times}\{0\}\in\Phi\}$. It is easy to see that the map $g$ is injective and order preserving.
% Put $\G_{\Phi}=\{G\subseteq X: G{\times}\{0\}\in\Phi\}$. Clearly, $\G_{\Phi}$ is a free open filter on $X$. 
Seeking a contradiction, assume that there exists $\Phi\in\mathbf{OF}(Y)$ such that $\F\setminus g({\Phi})\neq \emptyset$. Then, taking into account that the filter $\F$ is regular open, there exists a set $F=\Int(\overline{F})\in\F$ such that $F\notin g(\Phi)$.
%$G\cap (X\setminus F)\neq \emptyset$ for every $G\in g(\Phi)$.
Therefore $\overline{F}\notin g(\Phi)$, because otherwise we would have 
$F =\Int(\overline{F}) \in  g(\Phi)$, which we assumed to be not the case.
%We claim that $\overline{F}\notin g(\Phi)$. Indeed, otherwise, we would get that $\Int(\overline{F})=F\in g(\Phi)$, as the filter $g(\Phi)$ is open. But this contradicts our assumption. %Hence $H\cap ((X\setminus F){\times}\{0\})\neq \emptyset$ for any $H\in\Phi$. Since the filter $\Phi$ is open, $(\overline{F}\setminus F){\times}\{0\}\notin\Phi$. 
Hence $H\cap ((X\setminus \overline{F}){\times}\{0\})\neq \emptyset$ for each $H\in\Phi$. Let $\mathcal U$ be any free open ultrafilter on $Y$ such that $\Phi\cup\{(X\setminus \overline{F}){\times}\{0\}\}\subset \mathcal U$. It is easy to see that $(\mathcal U,1)\in Y$ is an accumulation point of the filter $\Phi$, which yields a contradiction. Hence $\F\subseteq g(\Phi)$ for every $\Phi\in\mathbf{OF}(Y)$. If $\G$ is a free open filter on $X$ such that $\F\subseteq \G$, then let $\Phi_\G$ be the filter on $Y$ generated by the family $\{G{\times}\{0\}:G\in\G\}$. Obviously, $\Phi_\G\in\mathbf{OF}(Y)$ and $g(\Phi_\G)=\G$, which completes the proof. 
\end{proof}

The following definition is crucial for constructing spaces with linear lattices of free open filters.

\begin{definition}\label{def}
Let $\F$ be an ultrafilter on an infinite cardinal $\kappa$. We write that $\mathfrak{h}(\F)> \alpha$ (and say that the {\em height} of $\F$ is greater than $\alpha$)  if there exists a scattered space $X\subset \beta(\kappa)$ satisfying the following conditions:
\begin{itemize}
    \item[(1)] 
    $\{\F\}= X^{(\alpha)}$ and for any $\xi\in \alpha$ the filter $\mathcal N(\F)$ traces on $X^{(\xi)}$ an ultrafilter $\F_{\xi}$;
    \item[(2)] for every $0<\gamma\leq \alpha$ and for each selector $\langle F_{\xi}\in \F_{\xi}: \xi\in \gamma\rangle$ such that the set $\bigcup_{\xi\in \gamma}F_{\xi}$ is open in $X$, there exists $F\in \F$ satisfying $\langle F\rangle \cap \bigcup_{\xi\in\gamma}X^{(\xi)}\subseteq \bigcup_{\xi\in \gamma}F_{\xi}$.
\end{itemize}
\end{definition}

For a given scattered
space $X$ witnessing that $\mathfrak h(\F) > \alpha$, and fixed $\xi < \alpha$, one may form the
scattered subspace 
$Y=\bigcup_{\gamma\in\xi}X^{(\gamma)}\cup\{\F\}$ of $X$, which witnesses that $\mathfrak h(\F) >\xi$. Hence the following is well defined: 
 $\mathfrak{h}(\F)=\alpha$ if $\mathfrak{h}(\F)>\xi$ for any $\xi\in\alpha$, but $\lnot (\mathfrak{h}(\F)>\alpha)$. 
 %Also note that if
%$\mathfrak h(\F) > \alpha$, then this can be witnessed by a scattered
%space $X$ such that $X^{(\alpha)} = \{\F\}$.

Clearly, $\mathfrak{h}(\F)=1$ if and only if $\F$ is a principal ultrafilter. The next lemma shows that consistently there exist ultrafilters of height 2.

\begin{lemma}
$\mathfrak{h}(\F)=2$ for each P-point $\F\in\omega^*$. 
\end{lemma}

\begin{proof}
To derive a contradiction, assume that there exists a scattered subspace $X\subset \beta(\omega)$ and a P-point $\F\in \omega^*$ such that $\F\in X^{(2)}$ and $\F$ traces on the set $X^{(1)}$ an ultrafilter which we denote by $\F_1$. Consider any countable family $\{H_n:n\in\omega\}\subset \F_1$. By the definition of $\F_1$, for every $n\in\omega$ there exists $F_n\in\F$ such that $\langle F_n\rangle \cap X^{(1)}\subset H_n$. Since $\F$ is a P-point, there exists $F\in\F$ such that $F\subset ^* F_n$ for any $n\in\omega$. Taking into account that $X^{(1)}$ consists of free ultrafilters, we obtain that $\langle F\rangle\cap X^{(1)}\subset \langle F_n\rangle\cap X^{(1)}\subset H_n$ for every $n\in\omega$. Thus, $\langle F\rangle\cap X^{(1)}\subset \bigcap_{n\in\omega}H_n$, witnessing that the ultrafilter $\F_1$ is $\omega_1$-complete. By~\cite[Lemma 10.2]{Jec}, the cardinality of the set $X^{(1)}$ is greater or equal to some measurable cardinal. On the other hand, $|X^{(1)}|\leq \mathfrak{c}$ which implies a contradiction.  
 \end{proof}

 Note that the above lemma does not immediately follow from the fact that P-points are not in the closure of any countable set, as $X^{(1)}$ can be uncountable.

For an ordinal $\alpha$, let $S^{*}(\alpha)=\alpha$ if $\alpha$ is finite and $S^{*}(\alpha)=\alpha+1$ otherwise. Let us note that $S^{*}(\alpha)$ equals the order type of the set $\{\delta: 0<\delta<\alpha+1\}$. 
The following two results justify Definition~\ref{def}.

\begin{proposition}\label{prneww}
Let $\F$ be an ultrafilter on $\kappa$ of height $>\alpha$ and $X$ be a scattered subspace of $\beta(\kappa)$, which witnesses $\mathfrak h(\F)>\alpha$. Then $\mathcal N(\F)$ traces on $Y=X\setminus\{\F\}$ a free regular open filter $\Phi$, and the subposet $\{\G\in \mathbf{OF}(Y): \Phi\subseteq \G\}$ of $\mathbf{OF}(Y)$ is order isomorphic to the ordinal $S^{*}(\alpha)$ endowed with the reversed order.
\end{proposition}

\begin{proof}
%Let $\Phi$ be the trace of $\mathcal N(\F)$ on $Y$. 
Since the space $X$ is regular and $\F\notin Y$, $\Phi$ is a free regular open filter. %By Proposition~\ref{supernew}, it is enough to check that the poset of all free open filters on $Y$ which contain $\Phi$ is order isomorphic to $(\alpha+1,\geq)$. 
For each nonzero ordinal $\delta<\alpha+1$ let $\G_\delta$ be the filter on $Y$ generated by the family $\Phi\cup\{\bigcup_{\xi\in\delta}X^{(\xi)}\}$. Since the space $Y$ is scattered, for each $0<\delta<\alpha+1$ the set $\bigcup_{\xi\in\delta}X^{(\xi)}$ is open. Thus, for every $0<\delta<\alpha+1$ the filter $\G_\delta$ is open. Since $\Phi$ is a free filter, we get that so are the filters $\G_\delta$, where $0<\delta<\alpha+1$.

Let $\G$ be an arbitrary free open filter on $Y$ such that $\Phi\subseteq \G$. Put $$\delta=\min\{\xi\in\alpha+1: \hbox{exists } G\in\G \hbox{ such that } G\cap X^{(\xi)}=\emptyset\}.$$ 
Since $X^{(0)}$ is an open dense subset of $Y$ and the filter $\G$ is open, $G\cap X^{(0)}\neq \emptyset$ for all $G\in\G$. Therefore, $\delta>0$.
We claim that $\G=\G_\delta$. The inclusion $\G_\delta\subseteq \G$ follows from the definition of $\delta$. In order to show the converse inclusion, fix any  $G\in\G$. Since the filter $\G$ is open and by the definition of $\delta$, we lose no generality assuming that the set $G$ is open and $G\subseteq \bigcup_{\xi\in\delta}X^{(\xi)}$. Since $\Phi\subseteq \G$, for every $\xi\in\delta$ the trace of $\G$ on $X^{(\xi)}$ is the ultrafilter $\F_{\xi}$ (see Definition~\ref{def}). Then for each $\xi\in\delta$ the set $F_\xi=G\cap X^{(\xi)}\in \F_{\xi}$. Since $\mathfrak{h}(\F)> \alpha$ (see condition (2) in Definition~\ref{def}), there exists $F\in\F$ such that $\langle F\rangle\cap \bigcup_{\xi\in\delta}X^{(\xi)}\subset G$. 
It remains to observe that $\langle F\rangle\in\mathcal N(\F)$, $\langle F\rangle\cap Y\in \Phi$ and 
$$G\supseteq \langle F\rangle\cap  \bigcup_{\xi\in\delta}X^{(\xi)}=\langle F\rangle\cap Y\cap \bigcup_{\xi\in\delta}X^{(\xi)}\in\G_\delta.$$ Thus, $G\in \G_\delta$ which implies that  $\G=\G_\delta$. Hence $$\{\G\in \mathbf{OF}(Y): \mathcal N(\F)\subseteq \G\}=\{\G_{\delta}:0<\delta<\alpha+1\}.$$
It is easy to see that the subposet $\{\G_{\delta}:0<\delta<\alpha+1\}\subset \mathbf{OF}(Y)$ is order isomorphic to $(S^{*}(\alpha),\geq)$. 
\end{proof}

Propositions~\ref{supernew} and~\ref{prneww} imply the following.

\begin{proposition}\label{prnew}
Let $\F$ be an ultrafilter on $\kappa$ of height $>\alpha$. Then there exists a scattered space $Y$ such that $\mathbf{OF}(Y)$ is order isomorphic to the ordinal $S^*(\alpha)$ endowed with the reversed order.
\end{proposition}

\begin{theorem}\label{thnew}
For each positive integer $n$ there exists an ultrafilter $\F$ on $\w$ such that $\mathfrak{h}(\F)> n$.
\end{theorem}

\begin{proof}
Fix any positive integer $n$. First we inductively construct a scattered subspace $X\subset\beta(\w)$. Let $X_0=\omega$. Assume that an infinite countable discrete subset $X_i\subset\beta(\omega)$ is already constructed for some $i< n-1$. Then put $X_{i+1}$ to be any countable infinite discrete subset of $\cl_{\beta(\w)}(X_i)\setminus X_i$. This way we construct the sets $X_i$ for every $i\leq n-1$. Fix any ultrafilter $\F\in \cl_{\beta(\w)}(X_{n-1})\setminus X_{n-1}$ and let $X$ be the subset $\bigcup_{i\in n}X_i\cup\{\F\}\subset \beta(\omega)$ endowed with the subspace topology. One can easily check that $X$ is scattered, $X^{(n)}=\{\F\}$ and $X^{(i)}=X_i$ for each $i<n$. Let us verify that $X$ is a witness for $\mathfrak{h}(\F)>n$. Fix any $i<n$. It is easy to check that each countable discrete subset of $\beta(\omega)$ is strongly discrete. Hence the set $X^{(i)}$ is strongly discrete.
By Lemma~\ref{d}, there exists a homeomoprhism $h:\overline{X^{(i)}}\to \beta(\omega)$ such that $h(X^{(i)})=\w$. Since $\F\in \cl_{\beta(\w)}(X^{(i)})$ we obtain that the filter $\mathcal N(\F)$ traces an ultrafilter on $X^{(i)}$, which we denote by $\F_i$. Hence condition (1) from Definition~\ref{def} is satisfied. In order to check condition (2) fix any $m\leq n$, and for each $i<m$ fix a set $F_i\in \F_i$. Since $\{\langle F\rangle: F\in\F\}$ is an open neighborhood base at $\F$, for each $i<m$ there exists $G_i\in\F$ such that $\langle G_i\rangle\cap X^{(i)}\subseteq F_i$. Then $\langle \bigcap_{i<m}G_i\rangle\cap \bigcup_{i<m}X^{(i)}\subseteq \bigcup_{i<m}F_i$, which fulfills condition (2) from Definition~\ref{def}. Hence $\mathfrak h (\F)>n$. 
\end{proof}

Proposition~\ref{prnew} and Theorem~\ref{thnew} imply the following:

\begin{theorem}\label{t1}
For each positive integer $n$ there exists a scattered space $X$ such that $\mathbf{OF}(X)$ is an $n$-element chain.
\end{theorem}

\begin{remark}
Theorem~\ref{t1} implies the affirmative answer to Questions~\ref{q1} and \ref{q2}. Namely, 
fix a positive integer $n$ and a space $X_n$ such that the poset $\mathbf{OF}(X_n)$ is order isomorphic to $n$. By Theorem~\ref{char}, the poset $\mathbf{H}(X_n)$ is order isomorphic to $n$ implying the affirmative answer to Question~\ref{q1}. In the proof of the implication (1) $\Rightarrow$ (2) in Theorem~\ref{char} we showed that each Hausdorff extension of an almost H-closed space has a singleton remainder. It follows that for each $Y\in\mathbf{H}(X_2)$ the set $Y\setminus X_2$ is singleton. This implies the affirmative answer to Question~\ref{q2}.  
\end{remark}

A subset $B$ of a scattered space $X$ is called {\em high} if $X^{(\alpha)}\cap B\neq \emptyset$ for every $\alpha\in ht(X)$.
 
\begin{theorem}\label{CH}
(CH) There exists an ultrafilter $\F$ on $\omega$ such that $\mathfrak h(\F)> \omega$. 
\end{theorem}

\begin{proof}
Observe that in the proof of Theorem~\ref{thnew}
for every $n\in\mathbb N$ we constructed a countable scattered subspace $X_n\subset\beta(\omega)$ which satisfies the following conditions:

\begin{itemize}
    \item  $ht(X_{n})=n+1$ and $X_{n}^{(n)}=\{\F_{n}\}$;
    \item  For each $m\in n$ the Cantor-Bendixson level $X_{n}^{(m)}$ is strongly discrete, implying that the filter $\mathcal N(\F_n)$ traces on $X_{n}^{(m)}$ an ultrafilter. 
\end{itemize}

Consider any partition $\{P_n:n\in\omega\}$ of $\omega$ into disjoint infinite subsets. For each $n\in\omega$, $X_{n}$ is homeomorphic to a subspace of $\langle P_{n}\rangle\cong \beta(\omega)$. 
Let $Y$ be a topological sum of $\{X_{n}:n\in\omega\}$. 
Identify $Y$ with a subspace of $\beta(\omega)$ such that $X_{n}\subset \langle P_{n}\rangle$.
%Since the sets $X_{n}$, $\xi\in \theta$ we get that $|Y_{\theta}|\leq\kappa$. %It is easy to see that $Y$
%Clearly, $\beta(\kappa)$ contains an isomorphic copy of the topological sum $\sqcup_{\xi\in\theta} \beta(|\xi|)$. So, similarly as above we identify the space $Y_{\theta}$ with a subspace of $\beta(\kappa)$.
%Consider the set of all sequences $\langle S_{\xi}: \xi \in\theta\rangle$ such that $S_{\xi}\in X_{\theta}'$ for any $\xi\in\theta$ and the set $\cup_{\xi\in\theta}S_{\xi}$ is open in $X_{\theta}'$.
Using CH, enumerate the set of all high open subsets of $Y$ as $\{U_{\alpha}:\alpha\in \omega_1\}$. Moreover, we assume that each high open subset of $Y$ appears cofinally many times in the enumeration. Also, enumerate all subsets of $\omega$ as $\{B_{\alpha}:\alpha\in\omega_1\}$. %Note that the latter two enumerations require CH. 
The desired ultrafilter $\F$ on $\omega$ will be constructed by recursion of length $\omega_1$. Fix any $\beta\in \omega_1$ and assume that we already constructed a family $\mathcal V_{\beta}=\{V_{\alpha}:\alpha\in\beta\}$ of subsets of $\omega$ which satisfies the following conditions:
\begin{itemize}
    \item[(a)] the family $\mathcal V_{\beta}$ is centered and, thus, generates a filter which we denote by $\mathcal W_{\beta}$; 
   % \item[(b)] for any $W\in\mathcal W_{\beta}$ the set $\langle W\rangle$ is high in $Y_{\theta}$;
    \item[(b)] for any $\alpha\in\beta$ either $B_{\alpha}\in \mathcal W_{\beta}$ or $\omega\setminus B_{\alpha}\in\mathcal W_{\beta}$;
    \item[(c)] for any $W\in \mathcal W_{\beta}$ the set $\langle W\rangle$ is high in $Y$, i.e., $\langle W\rangle \cap Y^{(m)}\neq \emptyset$ for any $m\in\omega$.
     %$K^{m}_W=\{n\in\omega: \langle W\rangle \cap X^{(m)}_{n}\neq \emptyset\}$ is infinite.
\end{itemize}
%It is easy to see that the set $K^{\alpha}_W$ is necessarily cofinal in $\theta$ for any $\alpha\in\theta$ and $W\in\mathcal W_{\beta}$.
The filter $\mathcal W_{\beta}$ is an approximation of the desired filter $\mathcal F$. At stage $\beta$ two cases are possible: 
\begin{itemize}
    \item[(i)] for any $n\in \omega$ there exists $W_{n}\in\mathcal W_{\beta}$ such that $\langle W_{n}\rangle\cap Y^{(n)}\subset U_{\beta}$; 
    \item[(ii)] there exists $n\in \omega$ such that $\langle W\rangle\cap (Y^{(n)}\setminus U_{\beta})\neq \emptyset$  for any $W\in\mathcal W_{\beta}$.
\end{itemize}

If case (ii) holds, then put $V_{\beta}=\omega$. Informally speaking the set $U_{\beta}$ is at this stage ``irrelevant'' for the filter $\mathcal W_{\beta}$.

Assume that case (i) holds. Since the ordinal $\beta$ is countable the filter $\mathcal W_{\beta}$ admits a countable nested base $\mathcal B=\{C_{n}:n\in\omega\}$, i.e., $C_n\subset C_m$ whenever $m\leq n$. Moreover, without loss of generality we can assume that $C_n\subset \bigcap _{i\leq n}W_i$ for any $n\in\omega$.

First we inductively construct an auxiliary strongly discrete set $E_{\beta}=\{e_{n}:n\in\omega\}\subset Y$ and a function $\phi\in\w^\w$. Assume that some $n\in \omega$ we already constructed a set $\{e_k: k\in n\}$ such that $e_k\in \langle C_k\rangle\cap X_{\phi(k)}^{(k)}$ and $\phi(k_1)\neq \phi(k_2)$ for any distinct $k_1,k_2\in n$. 
%Since the set $\langle C_n\rangle$ is high, for any $m\in\omega$ the set $\langle C_n\rangle\cap Y^{m}$ is infinite. 
Taking into account that the set $T=\bigcup_{i\in n}X_{\phi(i)}$ is not high and the set $\langle C_n\rangle$ is high, we get that there exists a point $e_n\in \langle C_n\rangle \cap (Y^{(n)}\setminus T)$. Let $k$ be the unique positive integer such that $e_n\in X_k$ and set $\phi(n)=k$. Recall that we identify $Y$ with the subspace of $\beta(\omega)$, so we consider points $e_n$, $n\in\omega$ as ultrafilters on $\omega$. For every $n\in\omega$ fix any element $S_n\in e_n$ such that $S_n\subset X_{\phi(n)}\cap C_n$ and $\langle S_n\rangle \cap Y^{(n)}=\{e_n\}$, which exists since the set $\langle C_n\rangle\cap X_{\phi(n)}\ni e_n$ is open in $Y$ and the set $Y^{(n)}$ is discrete. Note that the equality $\langle S_n\rangle\cap Y^{(n)}=\{e_n\}$ implies that $\langle S_n\rangle\cap Y\subset \bigcup_{i\leq n}Y^{(i)}$. Finally set $V'_{\beta}=\bigcup_{n\in\omega}S_{n}$. The choice of $S_{n}$, $n\in\omega$ together with the injectivity of the function $\phi$ ensure that %for any $y\in Y$ there exists at most one $n\in \omega$ such that $S_{n}\in y$. %The definition of the space $Y$ implies the following:
$$\langle V_{\beta}'\rangle\cap Y= \bigcup_{n\in\omega}(\langle S_{n}\rangle\cap Y)\subset \bigcup_{n\in\omega}(\langle C_n\rangle\cap \bigcup_{i\leq n}Y^{(i)})\subset \bigcup_{n\in\omega}(\langle \bigcap_{i\leq n} W_{i}\rangle\cap \bigcup_{i\leq n}Y^{(i)})\subset U_{\beta}.$$

%It can be checked that the definition of the function $\phi$ implies that for any $\alpha\in\kappa$ and $W\in\mathcal W_{\beta}$ the set $\{\xi\in\theta: \langle V_{\beta}'\cap W\rangle \cap X_{\xi}^{(\alpha)}\neq \emptyset\}$ has cardinality $\kappa$. 
\begin{claim}\label{claim*}
At least one of the following assertions holds:
\begin{itemize}
    \item[($\dagger$)] for any $W\in\mathcal W_{\beta}$ the set 
    $Z= \langle W\cap V_{\beta}'\cap B_{\beta}\rangle$ is high in $Y$;
    \item[($\dagger\dagger$)] for any $W\in\mathcal W_{\beta}$ the set
    $Z=\langle W\cap V_{\beta}'\cap (\omega\setminus B_{\beta})\rangle$ is high in $Y$.
\end{itemize}
\end{claim}
\begin{proof}
To derive a contradiction assume that both assertions fail. Then there exist $W_1,W_2\in\mathcal W_{\beta}$ and $n_1, n_2\in\omega$ such that for every $m\geq n=\max\{n_1,n_2\}$ the following equalities hold:
$$\langle W_1\cap V_{\beta}'\cap B_{\beta}\rangle\cap Y^{(m)}=\emptyset \quad \hbox{ and } \quad \langle W_2\cap V_{\beta}'\cap (\omega\setminus B_{\beta})\rangle\cap Y^{(m)}=\emptyset.$$ Put $W=W_1\cap W_2$. 
Since $$V_{\beta}'\cap W\subset (W_1\cap V_{\beta}'\cap B_{\beta})\cup (W_2\cap V_{\beta}'\cap (\omega\setminus B_{\beta}))$$ we obtain that for any $m\geq n$
$$\langle V_{\beta}'\cap W\rangle \cap Y^{(m)}\subset \langle W_1\cap V_{\beta}'\cap B_{\beta}\rangle\cap Y^{(m)}\cup \langle W_2\cap V_{\beta}'\cap (\omega\setminus B_{\beta})\rangle\cap Y^{(m)}=\emptyset.$$On the other hand, there exists $m\geq n$ such that $C_m\subset W$ and 
$$e_m\in \langle C_m\rangle\cap \langle V_{\beta}'\rangle\cap Y^{(m)}\subset\langle W\cap V_{\beta}'\rangle\cap Y^{(m)}=\emptyset,$$
which implies a contradiction. 
\end{proof}
If assertion ($\dagger$) of Claim~\ref{claim*} holds, then put $V_{\beta}=V'_{\beta}\cap B_{\beta}$. Otherwise, set $V_{\beta}=V'_{\beta}\cap (\omega\setminus B_{\beta})$. Clearly, the family $\mathcal V_{\beta+1}=\{V_{\alpha}:\alpha\in\beta+1\}$ satisfies the inductive hypothesis. So, after completing the induction we obtain a centered family $\mathcal V_{\omega_1}=\{V_{\alpha}:\alpha\in\omega_1\}$. Let $\F$ be the filter generated by the family $\mathcal V_{\omega_1}$. Fix any subset $B\subset\omega$. There exists $\xi\in\omega_1$ such that $B=B_{\xi}$. By the construction of the family $\mathcal V_{\omega_1}$, either $V_{\xi} \subset B_{\xi}$ or $V_{\xi}\subset \omega\setminus B_{\xi}$. Thus, either $B\in\F$ or $\omega\setminus B\in\F$, witnessing that $\F$ is an ultrafilter. Let $M=Y\cup\{\F\}$ be the subspace of $\beta(\omega)$. The definition of $\F$ (see condition (c)) implies that $M^{(\omega)}=\{\F\}$. By $\F_{n}$ we denote the trace of $\mathcal N(\F)$ on $M^{(n)}$. Observe that the definition of $Y$ implies that for each $n\in\omega$ the set $M^{(n)}=Y^{(n)}$ is strongly discrete. Hence $\F_n$ is an ultrafilter for every $n\in\omega$.
%as a topological sum of strongly discrete (possible empty) sets $X_{\alpha}^{(\xi)}$, $\alpha\in\theta$.

So, to prove that $\mathfrak{h}(\F)> \omega$ it remains to show that for any $0<\gamma\leq \omega$, for each selector $\langle F_{\xi}\in \F_{\xi}: \xi\in \gamma\rangle$ such that the set $U=\bigcup_{\xi\in \gamma}F_{\xi}$ is open in $M$, there exists $F\in \F$ satisfying $\langle F\rangle \cap (\bigcup_{\xi\in\gamma}M^{(\xi)})\subset U$. It is easy to see that for $\gamma<\omega$ the latter condition is automatically fulfilled, as the filter $\F$ is closed under finite intersections. Consider any selector $\langle F_{n}\in \F_{n}:n\in\omega\rangle$ such that the set $U=\bigcup_{n\in\omega}F_{n}$ is open in $M$. Since the set $U$ is open and high in $Y$ there exists a cofinal subset $\Xi\subset \omega_1$ such that $U=U_{\xi}$ for any $\xi\in\Xi$. For each $n\in \omega$ there exists a basic open neighborhood $\langle O_{n}\rangle$ of $\F$ which witnesses that $F_{n}\in \F_{n}$, that is, $\langle O_{n}\rangle\cap M^{(n)}= \langle O_{n}\rangle\cap Y^{(n)}\subset F_{n}$.  By the construction of $\F$ for each $n\in \omega$ there exists $m(n)\in\omega$ and a family $\{V_{\xi_0},\dots, V_{\xi_{m(n)}}\}\subset \mathcal V_{\omega_1}$ such that $\bigcap_{i\leq m(n)}V_{\xi_i}\subset O_n$. Set $\delta_n=\max\{\xi_0,\ldots, \xi_{m(n)}\}$. Since the set $\Xi$ is unbounded in $\omega_1$ there exists $\mu\in\Xi$ such that $\mu> \sup\{\delta_n:n\in \omega\}$. Then at stage $\mu$ of our induction the set $U=U_{\mu}$ will be already ``relevant'' for the filter $\mathcal W_{\mu}$, that is case (i) holds. Taking into account that $Y=\bigcup_{n\in\omega}M^{(n)}$, the inclusion $\langle V_{\mu}\rangle\cap Y\subset \langle V_{\mu}'\rangle\cap Y\subset U$ implies that $\mathfrak h(\F)> \omega$. 
\end{proof}

\begin{theorem}\label{measure}
If there exists an ultrafilter $\F$ such that $\mathfrak h(\F)> \omega+1$, then there exists a measurable cardinal.
\end{theorem}
    
    \begin{proof}
    Assume that there exists a cardinal $\kappa$ and an ultrafilter $\F\in\beta(\kappa)$ such that $\mathfrak{h}(\F)>\omega+1$. Then there exists a scattered subspace $X\subset \beta(\kappa)$ satisfying the following two conditions:
    \begin{itemize}
        \item[(1)] 
        $X^{(\w+1)}=\{\F\}$ and for any $\xi\in \w+1$ the
         filter $\mathcal N(\F)$ traces on $X^{(\xi)}$ an ultrafilter $\F_{\xi}$;
        \item[(2)] for any $0<\gamma\leq \w+1$ and for each selector $\langle F_{\xi}\in \F_{\xi}: \xi\in \gamma\rangle$ such that the set $\bigcup_{\xi\in \gamma}F_{\xi}$ is open in $X$, there exists $F\in \F$ satisfying $\langle F\rangle \cap \bigcup_{\xi\in\gamma}X^{(\xi)}\subseteq \bigcup_{\xi\in \gamma}F_{\xi}$.
    \end{itemize}
    To derive a contradiction, assume that there exists no measurable cardinal. Then the ultrafilters $\F$ and $\F_{\omega}$ are not $\omega_1$-complete, which implies the existence of families $\{H_{n}:n\in\omega\}\subset \F$ and  $\{G_n:n\in\omega\}\subset \F_{\omega}$ such that $\bigcap_{n\in\omega}H_n=\emptyset=\bigcap_{n\in\omega}G_n$.
    By the definition of $\F_{\omega}$, for every $n\in\omega$ there exists $T_n\in\F$ such that $\langle T_n\rangle \cap X^{(\omega)}\subset G_n$. For every $n\in\omega$ put $S_n=\bigcap_{i\leq n}H_n\cap \bigcap_{i\leq n}T_n$.
    Then $\{S_n: n\in\omega\}$ is a decreasing sequence of elements of $\F$ such that $\bigcap_{n\in\omega}S_n\subset \bigcap_{n\in\w}H_n=\emptyset$ and $\bigcap_{n\in\omega}(\langle S_n\rangle \cap X^{(\omega)})\subset \bigcap_{n\in\omega}G_n=\emptyset$. For any $n\in\omega$ put $F_n=\langle S_n\rangle\cap X^{(n)}\in\F_{n}$. Since $S_n\subset S_m$ whenever $m\leq n$, the set $\bigcup_{i\in\omega}F_i$ is open in $X$. Condition (2) implies the existence of a set $F\in \F$ such that $\langle F\rangle\cap \bigcup_{i\in\omega}X^{(i)}\subset \bigcup_{i\in\omega}F_i$. Put $G=\langle F\rangle\cap X^{(\omega)}\in\F_{\omega}$. 
    Since the set $\bigcup_{i\in\omega}F_i$ is open and $\langle F\rangle\cap \bigcup_{i\in\omega}X^{(i)}\subset \bigcup_{i\in\omega}F_i$, we get that the set $(\bigcup_{i\in\omega}F_i)\cup G$ is open as well. Taking into account that $\bigcap_{n\in\omega}(\langle S_n\rangle \cap X^{(\omega)})=\emptyset$ there exists $n\in\omega$ such that $W=G\setminus (\langle S_n\rangle \cap X^{(\omega)})\neq \emptyset$. 
    Consider any ultrafilter $\mathcal U\in W$. It follows that $S_n\notin \mathcal U$, witnessing that $\kappa\setminus S_n\in \mathcal U$. Since the set $(\bigcup_{i\in\omega}F_i)\cup G$ is open in $X$ there exists $U\in\mathcal U$ such that $\langle U\rangle \cap X\subset (\bigcup_{i\in\omega}F_i)\cup G$. Then $V=U\cap (\kappa\setminus S_n)\in\mathcal U$ and $\langle V\rangle\cap X \subset (\bigcup_{i\in\omega}F_i)\cup G$. Taking into account that $\mathcal U\in X^{(\omega)}$, $\emptyset\neq \langle V\rangle \cap X^{(n)}\subset F_n$. But since $V\cap S_n=\emptyset$, the definition of $F_n$ implies that $\langle V\rangle \cap F_n=\emptyset$. The obtained contradiction completes the proof.  
    %In order to show that the ultrafilter $\F_{\omega}$ is $\omega_1$-complete consider any countably subset $\{G_i\}_{i\in\omega}\subset F_{\omega}$. We can assume that $G_i\subset G_j$ whenever $j\leq i$. By the definition of $\F_{\omega}$, there exists a family $\{H_i\}_{i\in\omega}\subset \F$ such that $G_i=\langle H_i\rangle\cap X^{(\omega)}$ and $H_i\subset H_j$ whenever $j\leq i$. For any $n\in\omega$ put $F_n=\langle H_n\rangle\cap X^{(n)}\in\F_{n}$. Since $H_i\subset H_j$ whenever $j\leq i$, the set $\bigcup_{i\in\omega}F_i$ is open in $X$. Condition (2) implies the existence of a set $F\in \F$ such that $\langle F\rangle\cap (\bigcup_{i\in\omega}X^{(i)})\subset \bigcup_{i\in\omega}F_i$.
    %$$\langle F\rangle\cap (\bigcup_{i\in\omega}X^{(i)})=\bigcup_{i\in\omega}(\langle F\rangle \cap X^{(i)})\subset \bigcup_{i\in\omega}F_i=\bigcup_{i\in\omega}(\langle H_i\rangle\cap X^{(i)}).$$
    %We claim that $T=\langle F\rangle \cap X^{(\omega)}\subset \bigcap_{i\in\omega}G_i$. Indeed, fix any $x\in T$. 
    %Since the set $\Phi=\langle F\rangle \cap (X\setminus\{\F\})$ is open and $\Phi\subset \bigcup_{i\in\omega}F_i\cup T$ we get that the set $\bigcup_{i\in\omega}F_i\cup T$ is open as well. 
    \end{proof}
    
 Theorems~\ref{thnew},~\ref{CH} and~\ref{measure} yield the following natural problem.
    
    \begin{problem}
    Does there exist an ultrafilter $\F$ on $\w$ with $\mathfrak{h}(\F)>\omega$ in ZFC? What happens under MA?
    \end{problem}
    
    \begin{proposition}\label{pm}
    Let $\F$ be an ultrafilter on a measurable cardinal $\mu$ such that $\mathfrak{h}(\F)> \alpha\in\mu$. Then for every $n\in\omega$ there exists an ultrafilter $\F_n$ on $\mu$ such that $\mathfrak{h}(\F_n)> \alpha+n$.  
    \end{proposition}
    
    \begin{proof}
   The proof goes by induction on $n$.
    Assume that for some $n\in \omega$ there exists an ultrafilter $\F$ on a measurable cardinal $\mu$ such that $\mathfrak{h}(\F)> \alpha+n$.  We are going to find an ultrafilter $\mathcal W\in\beta(\mu)$ such that $\mathfrak{h}(\mathcal W)> \alpha+n+1$.
    Let $X$ be a scattered subspace of $\beta(\mu)$ which is a witness of $\mathfrak{h}(\F)> \alpha+n$. That is, 
    \begin{itemize}
        \item[(1)]  %$\F\in X^{(\alpha)}$ and for any $\xi\in \alpha$ the set $X^{(\xi)}$ is strongly discrete in $\beta(\kappa)$ which implies that the
        $\{\F\}= X^{(\alpha+n)}$ and for any $\xi\in \alpha+n$ the filter $\mathcal N(\F)$ traces on $X^{(\xi)}$ an ultrafilter $\F_{\xi}$;
        \item[(2)] for any $0<\gamma\leq \alpha+n$, for each selector $\langle F_{\xi}\in \F_{\xi}: \xi\in \gamma\rangle$ such that the set $\bigcup_{\xi\in \gamma}F_{\xi}$ is open in $X$ there exists $F\in \F$ satisfying $\langle F\rangle \cap \bigcup_{\xi\in\gamma}X^{(\xi)}\subseteq \bigcup_{\xi\in \gamma}F_{\xi}$.
    \end{itemize}
    Let $Y=X{\times}\mu$, where $\mu$ is endowed with the discrete topology. For every $\xi\in\mu$ by $X_{\xi}$ and $\F_{\xi}$ we denote the set $X{\times}\{\xi\}$ and the point $(\F,\xi)$, respectively. Being a topological sum of $\mu$-many disjoint copies of $X$, the space $Y$ can be identified with a scattered subspace of $\beta(\mu)$ such that the set $Y^{(\alpha+n)}=\{\F_{\xi}: \xi\in\mu\}$ is strongly discrete in $\beta(\mu)$.  By Lemma~\ref{d}, the set $\cl_{\beta(\mu)}(\{\F_{\xi}: \xi\in\mu\})$ is homeomorphic to $\beta(\mu)$. Then there exists an ultrafilter $\mathcal W$ on $\mu$ such that $\mathcal N(\mathcal W)$ traces on the set $\{\F_{\xi}: \xi\in\mu\}$ a $\mu$-complete ultrafilter. We claim that the scattered subspace $Z=Y\cup\{\mathcal W\}$ is a witness for $\mathfrak h(\mathcal W)> \alpha+n+1$. It is straightforward to check that $\mathcal W\in Z^{(\alpha+n+1)}$ and for every $\xi\in \alpha+n+1$ the trace of $\mathcal N(\mathcal W)$ on $Z^{(\xi)}$ (which we denote by $\mathcal W_{\xi}$) is an ultrafilter. Fix any $0<\gamma\leq \alpha+n+1$ and selector $\langle W_{\xi}\in \W_{\xi}: \xi\in \gamma\rangle$ such that the set $U=\bigcup_{\xi\in \gamma}W_{\xi}$ is open in $Z$. 
    %First we assume that $\gamma$  
    Then for each $\xi\in\gamma$ there exists a $A_{\xi}\in \W$ such that $\langle A_{\xi}\rangle \cap Z^{(\xi)}\subset W_{\xi}$. Since $|\gamma|\leq|\alpha|<\mu$ and the ultrafilter $\mathcal W_{\alpha+n}$ is $\mu$-complete, the set $V=\bigcap_{\xi\in\gamma} (\langle A_{\xi}\rangle\cap Z^{(\alpha+n)})\in \W_{\alpha+n}$. Let $D=\{\delta\in\mu: \F_{\delta}\in V\}$. For each $\xi\in\gamma$ and $\delta\in D$ the choice of $A_{\xi}$ implies that the set $W_{\xi}$ belongs to the trace of the filter $\mathcal N(\F_{\delta})$ on $Z^{(\xi)}$. For every $\delta\in\mu$ put $U_{\delta}=U\cap X_{\delta}$. Since the sets $U$ and $X_{\delta}$ are open in $Z$, for every $\delta\in \mu$ the set $U_{\delta}$ is open too. The definition of the space $Y$ implies that for each $\delta\in D$ and $\xi\in\gamma$ the set $U_{\delta}\cap Z^{(\xi)}$ belongs to the trace of the filter $\mathcal N(\F_{\delta})$ on $X^{(\xi)}_{\delta}$. Since for each $\delta\in D$ the space $X_{\delta}$ is a witness for $\mathfrak h(\F_{\delta})> \alpha+n$, there exists $F_{\delta}\in\F_{\delta}$ such that $\langle F_{\delta}\rangle\cap \bigcup_{\xi\in\gamma}X^{(\xi)}_{\delta}\subset U_{\delta}$. Then it is easy to check that $W=\bigcup_{\delta\in D}F_{\delta}\in \W$ and 
    $$\langle W\rangle \cap \bigcup_{\xi\in\gamma}Z^{(\xi)}=\bigcup_{\delta\in D}(\langle F_{\delta}\rangle \cap \bigcup_{\xi\in\gamma}Z^{(\xi)})\subset \bigcup_{\delta\in D}U_{\delta} \subset U.$$
    \end{proof}

    Proposition~\ref{prnew} and Theorem~\ref{CH} imply the following:
    
    \begin{corollary}\label{CCH}
    (CH) There exists a scattered space $X$ such that $\mathbf{OF}(X)$ is isomorphic to $(\omega+1,\geq)$.
    \end{corollary}
    
    %However, we don't know the answer to the following problem:
    
    Theorem~\ref{CH}, Proposition~\ref{prnew} and Proposition~\ref{pm} imply the following.
    
    %\begin{corollary}
    %(CH) For any measurable cardinal $\kappa$ and $n\in\omega$ there exists an ultrafilter $\F$ on $\kappa$ such that $\mathfrak h(\F)> \omega+n$.
    %\end{corollary}
    
    \begin{corollary}
    Assuming CH and the existence of a measurable cardinal, for every $n\in\omega$ there exists a scattered space $X$ such that $\mathbf{OF}(X)$ is isomorphic to $(\omega+n+1,\geq)$.
    \end{corollary}

    The following lemma justifies why we considered ordinals with the reversed order.
    
    \begin{lemma}
    Let $X$ be a scattered space. If $\mathbf{OF}(X)$ is order isomorphic to an ordinal $\alpha$, then $\alpha$ is finite.
    \end{lemma}
    
    \begin{proof}
    Let $X$ be a scattered space such that $\mathbf{OF}(X)$ is isomorphic to an ordinal $\alpha$. 
    By Theorem~\ref{char}, the space $X$ possesses a unique free open ultrafilter $\mathcal U$. Since $X^{(0)}$ is a dense open subspace of $X$, $X^{(0)}\in \mathcal U$. Since $X^{(0)}$ is discrete, the trace of the filter $\mathcal U$ on $X^{(0)}$ is an ultrafilter. Taking into account that $X^{(0)}\in\mathcal U$, it is easy to check that $\mathcal U$ is an ultrafilter on $X$. Let $\F_{\inf}$ be the infimum of $\mathbf{OF}(X)$.  
    Set 
    $$N=\min\{\xi\in ht(X):\hbox{ exists }F\in\F_{\inf} \hbox{ such that }F\cap X^{(\xi)}=\emptyset\}.$$ Since $X$ is scattered and the filter $\F_{\inf}$ is open, we get that $\bigcup_{i\in N}X^{(i)}\in\F_{\inf}$. It follows that the set $\bigcup_{i\in N}X^{(i)}$ belongs to each free open filter on $X$. If the ordinal $N$ is infinite, then for each $n\in \omega$ consider a free open filter $\F_n$ generated by the set $\F_{\inf}\cup\{\bigcup_{i\leq n}X^{(i)}\}$. Clearly, $\F_m\subsetneqq \F_{n}$ whenever $n<m$. Thus, the lattice $\mathbf{OF}(X)$ contains an infinite decreasing chain $\{\F_n:n\in\omega\}$, which contradicts our assumption. Hence $N\in\omega$.  
     To derive a contradiction, assume that there exists $i\in N$ such that $\F_{\inf}$ traces on $X^{(i)}$ the filter $\mathcal W_i$ which not an ultrafilter. Corollary~\ref{col} implies that the traces of $\mathcal U$ and $\F_{\inf}$ on $X^{(0)}$ coincides. Hence we get that $i>0$. The filter $\mathcal W_i$ can be enlarged to two different ultrafilters $\mathcal P$ and $\mathcal Q$. Then consider the filters $\F^{\mathcal P}$ and $\F^{\mathcal Q}$ generated by the families $\{A\cup B: A\in \F_{i-1}, B\in\mathcal P\}$ and $\{A\cup B: A\in \F_{i-1}, B\in\mathcal Q\}$, respectively. It is easy to see that $\F^{\mathcal P}$ and $\F^{\mathcal Q}$ are incomparable free open filters, which contradicts  the linearity of $\mathbf{OF}(X)$. Thus $\mathcal F_{\inf}$ traces an ultrafilter $\mathcal W_i$ on $X^{(i)}$ for every $i\in N$. 
    For every $i\in N$ let $\mathcal V_i$ be the ultrafilter on $X$ generated by the family $\mathcal W_i$.
    Since $\bigcup_{i\in N}X^{(i)}\in\F_{\inf}$, we get that $\F_{\inf}=\bigcap_{i\in N}\mathcal V_i$. It follows that there are only finitely many filters which contain $\F_{\inf}$. Hence the poset $\mathbf{OF}(X)$ is finite.
    \end{proof}

The following lemma implies that not every complete linear order can be represented as $\mathbf{OF}(X)$ for some space $X$. In particular, $\mathbf{OF}(X)$ cannot be order isomorphic to a dense linear order.
   
\begin{lemma}\label{gap}
If a poset $\mathbf{OF}(X)$ is linear, then for every distinct $a\leq b\in \mathbf{OF}(X)$ there exist distinct $c,d\in \mathbf{OF}(X)$ such that $a\leq c\leq d\leq b$ and there exists no $e\in \mathbf{OF}(X)\setminus\{c,d\}$ satisfying $c\leq e\leq d$.
\end{lemma}

\begin{proof}
Assume that the poset $\mathbf{OF}(X)$ is linear.
Fix any distinct free open filters $a\subset b$ on $X$. There exists $F\in b\setminus a$. Set $c=\sup\{y\in \mathbf{OF}(X): F\notin y\}$ and $d=\inf \{y\in \mathbf{OF}(X): F\in y\}$. By Theorem~\ref{char}, the filters $c$ and $d$ are well-defined. Lemma~\ref{op1} implies that $F\in d=\bigcap\{y\in \mathbf{OF}(X): F\in y\}$. By Corollary~\ref{uni}, $c=\bigcup\{y\in \mathbf{OF}(X): F\notin y\}$. It follows that $F\notin c$. Since $F\in d\setminus c$ and the poset $\mathbf{OF}(X)$ is linear we get that $c\leq d$. Clearly, $a\leq c\leq d\leq b$. Pick any $e\in \mathbf{OF}(X)$ and assume that $c\leq e\leq d$. Then either $F\in e$ or $F\notin e$. In the first case $e=d$. Otherwise, $e=c$.    
\end{proof}

Nevertheless, the following problem remains open:

\begin{problem}
Does there exist a Hausdorff space $X$ such that $\mathbf{OF}(X)$ is order isomorphic to an infinite ordinal?
\end{problem}

Note that Remark 1.7 from~\cite{PSV} implies that each non-compact regular space possesses at least $\omega_1$ free open filters.
So, after constructing spaces possessing arbitrary finite linear lattices of free open filters, it is natural to ask whether, for a given cardinal $\kappa$, there exists a space which possesses exactly $\kappa$ many free open filters. The following proposition gives the affirmative answer to this question.

\begin{proposition}
For each cardinal $\kappa$ there exists a space $Y$ which possesses exactly $\kappa$-many free open filters. 
\end{proposition}

\begin{proof}
If $\kappa<\omega$, then the statement follows from Theorem~\ref{t1}.
So, fix any infinite cardinal $\kappa$. For each $\alpha\in \kappa$ let $X_{\alpha}$ be a space admitting a unique free open filter $\mathcal{F}_{\alpha}$, which exists by Theorem~\ref{t1}. Let $Y$ be the disjoint union $\bigsqcup_{\alpha\in\kappa} X_{\alpha}\sqcup\{z\}$ endowed with the topology $\tau$ satisfying the following conditions:
\begin{itemize}
\item $X_{\alpha}$ is an open subspace of $(Y,\tau)$ for each $\alpha\in\kappa$;
\item open neighborhood base at $z$ consists of the sets $Y\setminus (\bigcup_{\alpha\in A}X_{\alpha})$, where $A\in [\kappa]^{<\omega}$.
%\item if $z\in U\in\tau$, then there exists a finite set $A\subset \kappa$ such that $Y\setminus U\subset \bigcup_{\alpha\in A}X_{\alpha}$.
\end{itemize}

For $A\in [\kappa]^{<\omega}$ let $\F_A$ be the filter on $Y$ generated by the family $\{\bigcup_{\alpha\in A}F_{\alpha}: F_{\alpha}\in\mathcal{F}_{\alpha}\}$. We claim that the set of all free open filters on the space $Y$ coincides with the set $\{\mathcal{F}_{A}:A\in[\kappa]^{<\omega}\}$ which has cardinality $\kappa$. By the definition of the topology on $Y$, for every $A\in[\kappa]^{<\omega}$ the filter $\F_A$ is open and free. 
%So, it remains to show the converse inclusion. 
Fix any free open filter $\mathcal{F}$ on the space $Y$. Since $\mathcal{F}$ is free there exists $F\in \mathcal{F}$ such that $z\notin \overline{F}$. It follows that there exists a finite subset $A\subset \kappa$ such that $F\subset \bigcup_{\alpha\in A}X_{\alpha}$. Let
$$B=\{\alpha\in\kappa:T\cap X_{\alpha}\neq\emptyset\hbox{  for each }T\in\mathcal{F}\}.$$
The arguments above imply that $B\subset A$. The set $B$ is nonempty, because otherwise there would exist elements $F_{\alpha}\in\mathcal{F}$, $\alpha\in A$ such that $F_{\alpha}\subset\bigcup_{\beta\in A}X_{\beta}$ and $F_{\alpha}\cap X_{\alpha}=\emptyset$ implying that $\emptyset=\bigcap_{\alpha\in A}F_{\alpha}\in \mathcal{F}$, which is impossible. Since for each $\alpha\in B$ the space $X_{\alpha}$ possesses the unique free open filter $\mathcal{F}_{\alpha}$ we obtain that for each $\alpha\in B$ the trace of the filter $\mathcal{F}$ on $X_{\alpha}$ coincides with $\mathcal{F}_{\alpha}$. Since the subspaces $X_{\alpha}$, $\alpha\in\kappa$ are clopen and pairwise disjoint it is straightforward to check that $\mathcal{F}=\mathcal{F}_{B}$.
\end{proof}

\section{Finite nonlinear lattices of free open filters}\label{5}

For each $n\in\omega$ let $\mathcal{F}_{n}$ be a filter on a set $X_{n}$. By $\prod_{i\in n}\mathcal{F}_i$
we denote the filter on the set $\prod_{i\in n}X_{i}$ generated by the family $\{\prod_{i\in n}F_{i}: F_{i}\in\mathcal{F}_{i}, i\in n\}$. If $n=2$, then the product of filters $\mathcal{F}_0$ and $\mathcal{F}_1$ is denoted  by $\mathcal{F}_0{\times}\mathcal{F}_1$.

%Recall that an ultrafilter $\mathcal{F}$ on a cardinal $\kappa$ is called {\em $\kappa$-complete} if for any $\lambda\in\kappa$ and subfamily $\{F_{\alpha}:\alpha\in\lambda\}\subset \mathcal{F}$ the set $\cap_{\alpha\in\lambda}F_{\alpha}$ belongs to $\mathcal{F}$.  %\js{I think that the~latter sentence should be restated.}

The following result was proved by Blass in his PhD-thesis.
\begin{lemma}\label{bl}
For ultrafilters $\mathcal{F}$ and $\mathcal{G}$ on sets $X$ and $Y$, respectively, the following conditions are equivalent:
\begin{enumerate}[\rm (i)]
\item $\mathcal{F}{\times}\mathcal{G}$ is an ultrafilter;
\item for every function $f:X\rightarrow \mathcal{G}$ there exists $F\in\mathcal{F}$ such that $\bigcap_{x\in F}f(x)\in\mathcal{G}$.
\end{enumerate}
\end{lemma}

Let $\{\kappa_{i}:1\leq i\leq n\}$ be an increasing sequence of measurable cardinals. For any $i\geq 1$ fix any free $\kappa_{i}$-complete ultrafilter $\mathcal{F}_{i}$ on $\kappa_{i}$.
Let $\kappa_0=\omega$ and $U_0$ be any free ultrafilter on $\omega$. For each $0<m\leq n$ by $\mathcal{U}_{m}$ we denote the filter
$\mathcal U_{m-1}{\times}\mathcal{F}_{m}$.
%Lemma~\ref{bl} implies the following.

\begin{corollary}\label{u}
For any $m\leq n$, $\mathcal{U}_{m}$ is an ultrafilter.
\end{corollary}

\begin{proof}
Assume that for some $m<n$, $\mathcal{U}_m$ is an ultrafilter. Observe that $|\prod_{i\leq m}\kappa_i|=\kappa_{m}<\kappa_{m+1}$ and the filter $\mathcal{F}_{m+1}$ is $\kappa_{m+1}$-complete. Therefore, Lemma~\ref{bl} implies that $\mathcal{U}_{m+1}=\mathcal{U}_m{\times}\mathcal{F}_{m+1}$ is an ultrafilter.
\end{proof}

\begin{lemma}\label{sd}
Let $\kappa$ be a measurable cardinal, $\mathcal {D}=\{x_{\alpha}:\alpha\in\kappa\}$ be a strongly discrete subset of $\beta(\kappa)$, $\phi:\overline{\mathcal {D}}\rightarrow \beta(\kappa)$ be a homeomorphism such that $\phi(x_{\alpha})=\alpha$ for each $\alpha\in\kappa$, and $x\in \overline{\mathcal {D}}\setminus \mathcal D$. If the ultrafilters $x_{\alpha}$, $\alpha\in \kappa$ are $\kappa$-complete and $\phi(x)$ is $\kappa$-complete, then $x$ is $\kappa$-complete.
\end{lemma}

\begin{proof}
Let us note that the homeomorphism $\phi$ exists by Lemma~\ref{d}. %Since $\phi$ is a homeomorphism we deduce that $\phi(\mathcal D)=\kappa$.
To derive a contradiction, assume that $x$ is not $\kappa$-complete. Then there exists a cardinal $\lambda<\kappa$ and a family $\{F_{\xi}:\xi\in \lambda\}\subset x$ such that $\bigcap_{\xi\in \lambda}F_{\xi}=\emptyset$. For each $\xi\in\kappa$ let $X_{\xi}=\{x_{\alpha}: F_{\xi}\in x_{\alpha}\}$. Since $x\in \overline{\{x_{\alpha}:\alpha\in\kappa\}}$ the sets $X_{\xi}$ are nonempty. Taking into account that $\phi$ is a homeomorphism and $\phi(x_{\alpha})=\alpha$ for each $\alpha\in\kappa$, it is easy to see that $\phi(X_{\xi})\in \phi(x)$ for each $\xi\in\lambda$. Since the filter $\phi(x)$ is $\kappa$-complete, the set $\bigcap_{\xi\in\lambda}\phi(X_{\xi})$ is nonempty and belongs to $\phi(x)$. Pick any
$$x_{\gamma}\in \phi^{-1}(\bigcap_{\xi\in\lambda}\phi(X_{\xi}))=\bigcap_{\xi\in\lambda}\phi^{-1}(\phi(X_{\xi}))=\bigcap_{\xi\in\lambda}X_{\xi}.$$
Note that $F_{\xi}\in x_{\gamma}$ for each $\xi\in \lambda$. Since the ultrafilter $x_{\gamma}$ is $\kappa$-complete, the set $\bigcap_{\xi\in\lambda}F_{\xi}$ belongs to $x_{\gamma}$, and hence it is nonempty, which contradicts the assumption.
\end{proof}

 The proof of the next lemma resembles the proof of Theorem~\ref{thnew}.
\begin{lemma}\label{ss}
Let $\kappa$ be a measurable cardinal. Then for each positive integer $n$ there exists a  $\kappa$-comple\-te ultrafilter $\U_n$ of $\kappa$ such that $\mathfrak h(\F_n)>n$. 
\end{lemma}

\begin{proof}
Fix a positive integer $n$.
Let us construct a scattered subspace $S_n$ of $\beta(\kappa)$ which witnesses that $\mathfrak h(\U_n)>n$.
Let $S_0=\kappa$. Assume that for some $i<n-1$ we constructed a strongly discrete set $S_i=\{x_{\alpha}:\alpha\in\kappa\}\subset\beta(\kappa)$ of cardinality $\kappa$ which consists of $\kappa$-complete ultrafilters. Next we shall construct a strongly discrete subset $S_{i+1}$. Decompose $\kappa$ into pairwise disjoint subsets $K_{\xi}$, $\xi\in\kappa$ such that $|K_{\xi}|=\kappa$ and $\bigcup_{\xi\in\kappa}K_{\xi}=\kappa$.
For each $\xi\in\kappa$ let $Y_{\xi}=\{x_{\alpha}:\alpha\in K_{\xi}\}$. Lemma~\ref{d} implies that for each $\xi\in\kappa$ there exists a homeomorphism $\phi_{\xi}: \overline{Y_{\xi}}\rightarrow \beta(\kappa)$ such that $\phi_{\xi}(Y_{\xi})=\kappa$.  Pick any free $\kappa$-complete ultrafilter $\F$ on $\kappa$. Let $S_{i+1}=\{\phi_{\xi}^{-1}(\F):\xi\in\kappa\}$. Since $S_i$ consists of $\kappa$-complete ultrafilters, Lemma~\ref{sd} implies that $S_{i+1}$ also consists of $\kappa$-complete ultrafilters. Let us show that $S_{i+1}$ is strongly discrete. Since the set $S_i$ is strongly discrete, there exist pairwise disjoint sets $F_{\alpha}\in x_{\alpha}$, $\alpha\in \kappa$. Note that for each $\xi\in\kappa$ the set $H_{\xi}=\bigcup_{\alpha\in K_{\xi}}F_{\alpha}$ belongs to $\phi_{\xi}^{-1}(\F)$. Since the sets $K_{\xi}$, $\xi\in\kappa$ are pairwise disjoint we get that the sets $H_{\xi}$ are pairwise disjoint as well, witnessing that the set $S_{i+1}$ is strongly discrete.
This way we construct the sets $S_i$ for $i\leq n-1$. Let $\F$ be a free $\kappa$-complete ultrafilter on $\kappa$ and $\phi: \overline{S_{n-1}}\rightarrow \beta(\kappa)$ be any homeomorphism such that $\phi(S_{n-1})=\kappa$ (which exists by Lemma~\ref{d}). Put $\U_n=\phi^{-1}(\F)$. Lemma~\ref{sd} ensures that the ultrafilter $\U_n$ is $\kappa$-complete. Let $S$ be the subspace $\bigcup_{i\in n}S_i\cup\{\U_n\}$ of  $\beta(\kappa)$. One can easily check that $S$ is scattered, $S^{(i)}=S_i$ for $i<n$ and $S^{(n)}=\{\U_n\}$. Similarly as in the proof of Theorem~\ref{thnew} one can check that the scattered space $S$ satisfies conditions (1) and (2) from Definition~\ref{def}, witnessing that  $\mathfrak h(\U_n)>n$.
\end{proof}

\begin{lemma}\label{plus}
Let $\F,\G$ be filters on a set $X$ and $\mathcal A$ a finite family of subsets of $X$ such that $\bigcup \mathcal A\in \F\cap \G$. If for each $A\in\mathcal A$ the traces of the filters $\F$ and $\G$ on $A$ coincide, then $\F=\G$.
\end{lemma}

\begin{proof}
Fix any $F\in \F$ and for each $A\in \mathcal A$ let $F_A=F\cap A$. By the assumption, for every $A\in\mathcal A$ there exists $\bigcup\mathcal A\supseteq G_A\in\G$ such that $G_A\cap A\subseteq F_A$. Then $\bigcap_{A\in\mathcal A}G_A\in\G$ and $(\bigcap_{A\in\mathcal A}G_A) \subseteq \bigcup_{A\in\mathcal A}F_A\subseteq F$,
witnessing that $\F\subseteq \G$. Similarly,  one can show the converse inclusion.
\end{proof}

\begin{theorem}\label{t2}
Let $\{\kappa_i:1\leq i\leq n-1\}$ be an increasing sequence of measurable cardinals.  Then for each sequence $\{m_i: i\in n\}$ of positive integers there exists a Hausdorff space $X$ such that the lattice $\mathbf{OF}(X)$ is order isomorphic to $\prod_{i\in n}m_i$.
\end{theorem}

\begin{proof}
Let $\kappa_0=\w$.
By Theorem~\ref{thnew} there exists an ultrafilter $\F_0$ on $\w$ of height $>m_0$. Let $S_0$ be a scattered subspace of $\beta(\w)$, which witnesses $\mathfrak h(\F_0)>m_0$, i.e. $S_0$ satisfies conditions (1) and (2) from Definition~\ref{def}.
By Lemma~\ref{ss}, for each $1\leq i< n$ there exists a $\kappa_i$-complete ultrafilter $\F_i$  on $\kappa_i$ of height $>m_i$. Let $S_i$ be a scattered subspace of $\beta(\kappa_i)$ which witnesses  $\mathfrak{h}(\F_i)>m_i$.  
By the proofs of Theorem~\ref{thnew} and Lemma~\ref{ss}, without loss of generality we can assume that $\kappa_i=S_i^{(0)}$ for each $i\in n$.

By Proposition~\ref{supernew}, it suffices to construct a space $Z$ and a free regular open filter $\F$ on $Z$ such that the subposet $\{\G\in\mathbf{OF}(Z): \F\subseteq \G\}$ of $\mathbf{OF}(Z)$ is order isomorphic to $\prod_{i\in n}m_i$. For each $i\in n$ put
$$E_i=\{\F_0\}{\times}\dots {\times}\{\F_{i-1}\}{\times}S_i{\times}\{\F_{i+1}\}{\times}\dots{\times}\{\F_{n-1}\}.$$
Let $Y$ be the subspace $(\bigcup_{i\in n}E_i)\cup\prod_{i\in n}\kappa_i$ of the Tychonoff product $\prod_{i\in n}S_i$. Also, let $\Phi=(\F_0,\F_1,\ldots,\F_{n-1})$, $Z=Y\setminus \{\Phi\}$, and $\F$ be the trace of $\mathcal N(\Phi)$ on $Z$. 
Observe that the set $\prod_{i\in n}\kappa_i$ is dense in $Z$ and consists of isolated points. It is straightforward to check that for each $i\in n$ the set $(E_i\setminus \{\Phi\}) \cup \prod_{i\in n}\kappa_i$ is open and dense in $Z$.

Since for each $i\in n$ the space $S_i$ is regular, we get that $Y$ is regular as well. It follows that $\F$ is a free regular open filter. For each $i\in n$ let $\mathbb P_i=\{\G\in\mathbf{OF}(S_i): \mathcal N(\F_i)\subseteq \G\}$. 
The proof of Proposition~\ref{prneww} implies that for each $i\in n$, $\mathbb P_i=\{\G^{i}_k:0<k\leq m_i\}$, where $\G^{i}_k$ is generated by $\mathcal N(\F_i)\cup\{\bigcup_{j\in k}{S_i^{(j)}}\}$. It follows that for each $i\in n$ the set of all free open filters on $E_i$, which contain the trace of $\F$ on $E_i$,
coincides with the set $\{\H^{i}_k:0< k\leq m_i\}$, where 
$$\H^i_k=\{\{\F_0\}{\times}...{\times}\{\F_{i-1}\}{\times}G{\times}\{\F_{i+1}\}...{\times}\{\F_{n-1}\}: G\in \G^i_k\}.$$

For each $\vec{v}=(v_0,\ldots,v_{n-1})\in \prod_{i\in n}m_i$ let $T_{\vec{v}}=\{i\in n: v_i>0\}$ and $\mathcal W_{\vec{v}}$ be the filter on $Z$ generated by the family $$\{A\cup \bigcup_{i\in T_{\vec{v}}}B_i: A\in\prod_{i\in n}\F_i \hbox{ and } B_i\in \mathcal H^i_{v_i}\}.$$
It is clear that $\F\subseteq \mathcal W_{\vec{v}}$ for all $\vec{v}\in\prod_{i\in n}m_i$. Therefore, the filters $\mathcal W_{\vec{v}}$, $\vec{v}\in\prod_{i\in n}m_i$ are free. 
Taking into account that for each $i\in n$ the set $E_i\cup\prod_{i\in n}\kappa_i$ is open in $Z$, it is routine to check that $\mathcal W_{\vec{v}}$ is an open filter for each $\vec{v}\in \prod_{i\in n}m_i$.

Fix any free open filter $\Psi$ on $Z$ which contains $\F$. Since $\Psi$ is an open filter, the open dense set $\prod_{i\in n}\kappa_i$ is positive with respect to $\Psi$. Recall that $\F$ traces on $\prod_{i\in n}\kappa_i$ the ultrafilter $\prod_{i\in n}\F_i$. Since $\F\subseteq \Psi$, we get that $\Psi$ also traces on $\prod_{i\in n}\kappa_i$ the ultrafilter $\prod_{i\in n}\F_i$. Let 
$$Q=\{i\in n: E_i \hbox { is positive with respect to }\Psi\}.$$ 
Since $\F\subseteq \Psi$, for each $i\in Q$ there exists $0<k_i\leq m_i$ such that the trace of $\Psi$ on $E_i$ coincides with the filter $\mathcal H_{k_i}^i$. Lemma~\ref{plus} implies that $\Psi$ coincides with the filter $\mathcal W_{\vec{v}}$, where $v_i=k_i$ if $i\in Q$ and $v_i=0$, otherwise. 

Hence $\mathbf{OF}(X)=\{\mathcal W_{\vec{v}}:\vec{v}\in\prod_{i\in n}m_i\}$. At this point it is easy to see that the lattice $\mathbf{OF}(X)$ is order isomorphic to $\prod_{i\in n}m_i$.
\end{proof}

Theorem~\ref{t1} implies that for every lattice $L$ with $|L|\leq 3$ there is a space $X$ such that $\mathbf{OF}(X)$ is order isomorphic to $L$.
Theorems~\ref{t1} and~\ref{t2} provide that the existence of a measurable cardinal implies that for each lattice $L$ of cardinality $4$, there is a space $X$ such that $\mathbf{OF}(X)$ is order isomorphic to $L$.
Theorem~\ref{char} yields the existence of two (non-distributive) five-element lattice which are not isomorphic to $\mathbf{OF}(X)$ for any space $X$.
The next example shows that assuming the existence of two  measurable cardinals for any five-element distributive lattice $L$ there exists a space $X$ such that $\mathbf{OF}(X)$ is isomorphic to $L$. Let $L_0=\{(0,0),(1,1),(1,2),(2,1),(2,2)\}$ and $L_1=\{(0,0),(1,0),(0,1),(1,1),(2,2)\}$  be the sublattices of $\omega^2$ displayed below.
%Let $L_0$ and $L_1$ be the lattices displayed below. 

\begin{figure}[H]
\begin{center}
\begin{tikzpicture}[]
\node[circle,draw] (a) at (0, 0) {};
\node[circle,draw] (b) at (1, 1) {};
\node[circle,draw] (c) at (1, 2) {};
\node[circle,draw] (d) at (2, 1) {};
\node[circle,draw] (e) at (2, 2) {};
\node[circle,draw] (f) at (7, 0) {};
\node[circle,draw] (g) at (8, 0) {};
\node[circle,draw] (h) at (7, 1) {};
\node[circle,draw] (i) at (8, 1) {};
\node[circle,draw] (j) at (9, 2) {};
\foreach \from/\to in {a/b,b/c,b/d,c/e,d/e,f/g,f/h,g/i,i/h,j/i} \draw [-] (\from) -- (\to);
\end{tikzpicture}
\end{center}

\begin{minipage}{0.35\textwidth}
\subcaption{\label{left} Lattice $L_0$}
\end{minipage}
\begin{minipage}{0.35\textwidth}
\subcaption{\label{right} Lattice $L_1$}
\end{minipage}

%\caption{Lattices $L_0$ (left) and $L_1$ (right).}
\label{diagram:2}
\end{figure}

\begin{example}\label{ex1}
Let $\kappa_1<\kappa_2$ be measurable cardinals. Then for any $i\in 2$ there exists a space $Y_i$ such that the lattice $\mathbf{OF}(Y_i)$ is isomorphic to $L_i$.
\end{example}

\begin{proof}
Set $\kappa_0=\omega$ and
for every $i\in 3$ fix any $\kappa_{i}$-complete ultrafilter $\mathcal{F}_{i}$ on $\kappa_{i}$. Let $X_{i}=\kappa_i\cup\{\F_i\}$ be the subspace of $\beta(\kappa_i)$. For convenience we denote the point $(\F_0,\F_1,\F_2)\in \prod_{i\in 3}X_i$ by $x^*$.

{\bf Construction of $Y_0$}.
Let $X$ be the subspace $(X_0{\times}X_1{\times}\kappa_2)\cup\{x^*\}$ of the Tychonoff product $\prod_{i\in 3}X_{i}$.
Put $Y_0=\mathrm{M}(X,x^*)$ (see Construction~\ref{cons}). One can check that $\mathbf{OF}(Y_0)=\{\Phi_{(i,j)}:(i,j)\in L_0\}$, where the free open filters $\Phi_{(i,j)}$ are defined as follows:
\begin{itemize}
\item $\Phi_{(2,2)}$ is generated by the family $\{(\prod_{i\in 3}F_i){\times}\{0\}:F_i\in\F_i\}$;
\item $\Phi_{(1,2)}$ is generated by the family $\{(F_0\cup\{\F_0\}){\times}F_1{\times}F_2{\times}\{0\}:F_i\in\F_i\}$;
\item $\Phi_{(2,1)}$ is generated by the family $\{F_0{\times}(F_1\cup\{\F_1\}){\times}F_2{\times}\{0\}:F_i\in\F_i\}$;
\item $\Phi_{(1,1)}$ is generated by the family $\{U\cup V: U\in \Phi_{(1,2)}$ and $V\in \Phi_{(2,1)}\}$;
\item $\Phi_{(0,0)}$ is generated by the family $\{(F_0\cup\{\F_0\}){\times}(F_1\cup\{\F_1\}){\times}F_2{\times}\{0\}:F_i\in\F_i\}$.
\end{itemize}
Moreover, the map $\phi:\mathbf{OF}(Y_0)\rightarrow L_0$, $\phi(\Phi_{(i,j)})=(i,j)$ is an order isomorphism.

\smallskip

{\bf Construction of $Y_1$}.
Let $X$ be the subspace
$$(X_0{\times}\kappa_1{\times}\kappa_2)\cup(\{\F_0\}{\times}\{\F_1\}{\times}\kappa_2)\cup (\{\F_0\}{\times}\kappa_1{\times}\{\F_2\})\cup \{x^*\}$$of the Tychonoff product $\prod_{i\in 3}X_{i}$.
Put $Y_1=\mathrm{M}(X,x^*)$. One can check that $\mathbf{OF}(Y_1)=\{\Phi_{(i,j)}:(i,j)\in L_1\}$, where the free open filters $\Phi_{(i,j)}$ are defined as follows:
\begin{itemize}
\item $\Phi_{(2,2)}$ is generated by the family $\{(\prod_{i\in 3}F_i){\times}\{0\}:F_i\in\F_i\}$;
\item $\Phi_{(1,1)}$ is generated by the family $\{(F_0\cup\{\F_0\}){\times}F_1{\times}F_2{\times}\{0\}:F_i\in\F_i\}$;
\item $\Phi_{(0,1)}$ is generated by the family $\{U\cup (\{\F_0\}{\times}\{\F_1\}{\times}F_2{\times}\{0\}): U\in \Phi_{(1,1)}$ and $F_2\in\F_2\}$;
\item $\Phi_{(1,0)}$ is generated by the family $\{U\cup (\{\F_0\}{\times}F_1{\times}\{\F_2\}{\times}\{0\}): U\in \Phi_{(1,1)}$ and $F_1\in\F_1\}$;
\item $\Phi_{(0,0)}$ is generated by the family $\{U\cup V: U\in \Phi_{(1,0)}$ and $V\in \Phi_{(0,1)}\}$.
\end{itemize}
Moreover, the map $\phi:\mathbf{OF}(Y_1)\rightarrow L_1$, $\phi(\Phi_{(i,j)})=(i,j)$ is an order isomorphism.
\end{proof}

A subset $A$ of a poset $P$ is called an {\em antichain} if $a\nleq b$ for any two distinct elements $a,b\in A$. The next proposition shows that it is easy to construct a space $X$ whose
lattice of free open filters contains large chains and antichains. Also, it reveals why we cannot use Katetov extension instead of Mooney's scheme during the construction of spaces with finite or linear lattices of free open filters.

\begin{proposition}\label{L}
For every cardinal $\kappa$ there exists an almost H-closed space $X$ such that the lattice $\mathbf{OF}(X)$ contains a chain and an antichain of cardinality $>2^\kappa$.
\end{proposition}

\begin{proof}
Let $\kappa$ be any infinite cardinal and $K(\kappa)$ be the Katetov extension of the discrete space $\kappa$. Fix an ultrafilter $\mathcal U$ on $\kappa$ such that every element of $\mathcal U$ has cardinality $\kappa$. For instance, one can consider any ultrafilter which contains the filter $\{S\subset \kappa: |\kappa\setminus S|<\kappa\}$. By $X$ we denote the subspace $K(\kappa)\setminus \{\mathcal U\}$ of $K(\kappa)$. Since $\kappa$ is a dense discrete subspace of $X$, $\mathcal U$ is the unique free open ultrafilter on $X$. Consider the coarsest free open filter $\F_{\inf}$ on $X$. Corollary~\ref{col} implies that $\F_{\inf}$ is generated by the family $\{\Int(\overline{U}): U\in \mathcal U\}$. Put $X^*=X\setminus \kappa$ and let $\mathcal G$ be the trace of the filter $\F_{\inf}$ on the set $X^*$. Clearly, the character of $\mathcal U$ is at most $2^{\kappa}$. It follows that the character of the filter $\mathcal G$ is $\leq 2^{\kappa}$ as well. It is easy to see that the closure of every subset $A\in[\kappa]^{\kappa}$ in $K(\kappa)$ is open and homeomorphic to $K(\kappa)$. Since $|K(\kappa)|> 2^{\kappa}$, we get that each element of $\F_{\inf}$ has cardinality $>2^{\kappa}$. It follows that every element $G\in \mathcal G$ has cardinality $>2^{\kappa}$.
The latter two facts imply that $\mathcal G$ is not an ultrafilter on $X^*$ and the set $\mathbf{A}$ of all ultrafilters on $X\setminus \kappa$ which contain $\mathcal G$ has cardinality $>2^{\kappa}$. It is straightforward to check that for each $\mathcal F\in\mathbf{A}$ the filter 
$$\mathcal T_{\F}=\{F\cup U: F\in \F, U\in \mathcal U\}$$ on $X$ is open and free. Then the family $\{\mathcal T_{\F}:\F\in\mathbf{A}\}\subset\mathbf{OF}(X)$ forms an antichain of size $>2^{\kappa}$. 

Pick any $\F\in\mathbf{A}$ such that $\{S\subset X^*: |X^*\setminus S|\leq 2^\kappa\}\subset \F$, which exists as every element of $\mathcal G$ has cardinality $> 2^\kappa$. Let $\lambda$ be the character of $\F$.
Taking into account that $\lambda> 2^\kappa$ and the character of $\G\leq 2^{\kappa}$, one can inductively construct a family $\{B_{\alpha}:\alpha\in \lambda\}\subset \F$ such that for each $\alpha\in\lambda$, $B_{\alpha}$ does not belong to the filter generated by the family $\mathcal G\cup\{B_{\xi}:\xi\in\alpha\}$.
%Fix any base $\mathcal B_{\G}$ of $\G$ of cardinality $\leq 2^\kappa$. For every $\alpha\in\lambda$ there exists $\theta(\alpha)$ such that $B_{\theta(\alpha)}$ does not belong to the filter generated by the family $\mathcal B_{\G}\cup\{B_{\xi}:\xi\in\alpha\}$. The ordinal $\theta(\alpha)$ is well-defined, because $|\mathcal B_{\G}\cup\{B_{\xi}:\xi\in\alpha\}|<\lambda$. Let $\delta(\alpha)=\sup\{\theta(\alpha), \theta(\theta(\alpha)),\ldots, \theta^{n}(\alpha),\ldots\}$. Clearly, $B_{\delta(\alpha)}$ does not belong to the filter generated by the family $\mathcal B_{\G}\cup\{B_{\xi}:\xi\in\delta(\alpha)\}$. %Since $\lambda$ is regular, the family $\{\delta(\alpha):\alpha\in \lambda\}$ is of cardinality $\lambda$.
For each $\alpha\in\lambda$ let $\H_{\alpha}$ be the filter on $X^*$ generated by the set $\mathcal G\cup \{B_{\xi}:\xi\in \alpha\}$. Then for every $\alpha\in \lambda$, the filter $$\F_{\alpha}=\{H\cup U: H\in \H_{\alpha}, U\in\mathcal U\}$$ on $X$ is free and open. Moreover, $\F_{\alpha}\subsetneqq \F_{\beta}$ if and only if $\alpha\in\beta$. Hence the set $\{F_{\alpha}:\alpha\in \lambda\}\subset \mathbf{OF}(X)$ is a chain of cardinality $\lambda>2^\kappa$.
\end{proof}

Results of this chapter yield the following natural problem.

\begin{problem}
Does there exist in ZFC a Hausdorff space $X$ such that $\mathbf{OF}(X)$ is a finite nonlinear lattice? 
\end{problem}

%A {\em character} of a filter $\F$ is the smallest cardinal $\kappa$ such that $\F$ possesses a base of size $\kappa$. 

\section*{Acknowledgements} We would like to thank the referee for many valuable comments and suggestions, which essentially improved the presentation of the paper.

\end{document}